\documentclass[]{article}

\usepackage{graphicx}
\usepackage{amsmath}
\usepackage{amssymb}
\usepackage{amsthm} 
\usepackage{bm}
\usepackage{pxfonts}
\usepackage{enumerate}
\usepackage{color}
\usepackage{mathdots}
\usepackage{sectsty}
\usepackage[hidelinks]{hyperref}
\usepackage{tikz}
\usepackage{caption}
\usepackage{adjustbox}
\usepackage{bbold}
\usepackage{tikz}
\usepackage{genyoungtabtikz}
\allowdisplaybreaks

\sectionfont{\scshape\centering\fontsize{11}{14}\selectfont}
\subsectionfont{\scshape\fontsize{11}{14}\selectfont}
\usepackage{fancyhdr}

\newcommand\shorttitle{Convergence and explicit formula for joint moments of CJ$\beta$E characteristic polynomial}
\newcommand\authors{Theodoros Assiotis, Mustafa Alper Gunes and Arun Soor}

\newtheorem{thm}{Theorem}[section]
\newtheorem{cor}[thm]{Corollary}
\newtheorem{lem}[thm]{Lemma}
\newtheorem{defn}[thm]{Definition}
\newtheorem{rmk}[thm]{Remark}
\newtheorem{prop}[thm]{Proposition}

\newtheorem*{claim*}{Claim}

\title{\large \bf CONVERGENCE AND AN EXPLICIT FORMULA FOR THE JOINT MOMENTS OF THE CIRCULAR JACOBI $\beta$-ENSEMBLE CHARACTERISTIC POLYNOMIAL}
\author{\small THEODOROS ASSIOTIS, MUSTAFA ALPER GUNES AND ARUN SOOR}

\date{}

\begin{document}

\maketitle

\begin{abstract}
The problem of convergence of the joint moments, which depend on two parameters $s$ and $h$, of the characteristic polynomial of a random Haar-distributed unitary matrix and its derivative, as the matrix size goes to infinity, has been studied for two decades, beginning with the thesis of Hughes \cite{HughesThesis}. Recently, Forrester \cite{forrester2020joint} considered the analogous problem for the Circular $\beta$-Ensemble (C$\beta$E) characteristic polynomial, proved convergence and obtained an explicit combinatorial formula for the limit for integer $s$ and complex $h$. In this paper we consider this problem for a generalisation of the C$\beta$E, the Circular Jacobi $\beta$-ensemble (CJ$\beta\textnormal{E}_\delta$), depending on an additional complex parameter $\delta$ and we prove convergence of the joint moments for general positive real exponents $s$ and $h$. We give a representation for the limit in terms of the moments of a family of real random variables of independent interest. This is done by making use of some general results on consistent probability measures on interlacing arrays. Using these techniques, we also extend Forrester's explicit formula to the case of real $s$ and $\delta$ and integer $h$. Finally, we prove an analogous result for the moments of the logarithmic derivative of the characteristic polynomial of the Laguerre $\beta$-ensemble.
\end{abstract}

\section{Introduction}

Let $\mathbb{U}(N)$ be the group of $N\times N$ unitary matrices and consider $\mathbf{V}\in \mathbb{U}(N)$. Introduce the quantities, where $e^{\textnormal{i} \theta_1},\dots, e^{\textnormal{i}\theta_N}$ are the eigenvalues of $\mathbf{V}$: 
\begin{equation}\label{CharPolyDef}
\begin{aligned}
    \Phi_{\mathbf{V}}(u) := \prod_{j=1}^N\left(1-e^{\textnormal{i}(u-\theta_j)}\right), && \Psi_{\mathbf{V}}(u) := e^{\frac{1}{2}\textnormal{i}\left(N(u+\pi) + \sum_{j=1}^N \theta_j\right)}\Phi_\mathbf{V}(u).
\end{aligned}
\end{equation}
Namely, $\Phi_{\mathbf{V}}$ is the characteristic polynomial of $\mathbf{V}$ and $\Psi_{\mathbf{V}}$ is chosen to satisfy $\left|\Psi_{\mathbf{V}}(u) \right| = |\Phi_\mathbf{V}(u)|$ and $\Psi_\mathbf{V}(u) \in \mathbb{R}$ whenever $u\in \mathbb{R}$. Let $d\mathsf{Haar}$ denote the Haar probability measure on $\mathbb{U}(N)$ and consider the following joint moments, which exist for $-\frac{1}{2}<h<s+\frac{1}{2}$:
\begin{equation}
    F_{N}(s,h) := \int_{\mathbb{U}(N)} \left|\Psi_\mathbf{V}(0)\right|^{2s-2h} \left|\frac{d}{du}\Psi_\mathbf{V}(u)\bigg|_{u=0}\right|^{2h}d\mathsf{Haar}(\mathbf{V}).
\end{equation}

There has been significant interest in the problem of asymptotics of $F_N(s,h)$, as $N\to \infty$, for the past twenty years, beginning with the thesis of Hughes from 2001 \cite{HughesThesis}. See \cite{HughesThesis,ConreyRubensteinSnaith,Dehaye2008,Dehaye2010note,Winn_2012,7authors,Bailey_2019,assiotis2020joint,Distinguished} for a number of different approaches (of analytic, combinatorial or probabilistic nature) to this problem that have been developed through the years. Part of the initial motivation of the thesis of Hughes for studying the asymptotics of these moments comes from a remarkable conjectural connection to the joint moments of Hardy's function from analytic number theory, see \cite{HughesThesis,HallHardysFunction,ConreyHardysFunction}. More recently, a connection to random unitarily invariant infinite Hermitian matrices was understood in \cite{assiotis2020joint}, see also \cite{Distinguished}. Finally, these moments are also closely related to the theory of integrable systems, in particular Painlev\'{e} equations, see \cite{ForresterWitte2006,7authors,Bailey_2019,assiotis2020joint,Distinguished}.

Before continuing, we note that it turns out that the study of these moments is equivalent to the study of the moments of the sum of points of certain determinantal point processes (or equivalently the trace of certain unitarily invariant random Hermitian matrices), see \cite{ForresterBook} for the definitions. Such problems have their own intrinsic interest and have been studied for a long time \cite{ForresterBook}. For example, it has been shown that for special cases of determinantal processes the Laplace transform of the sum of points is connected to integrable systems \cite{ChenIts,BasorChenEhrhardt} (it can be instructive to think of this quantity in analogy to the gap probability of a determinantal process \cite{ForresterBook} as both are given by expectations of very simple multiplicative functionals and in particular as Fredholm determinants).

Now, returning to our problem, it is well-known, see \cite{ForresterBook}, that the distribution of the eigenangles $\boldsymbol{\theta}=(\theta_1,\dots,\theta_N)$ of a Haar distributed matrix from $\mathbb{U}(N)$ is explicit and given by the probability measure on $[0,2\pi]^N$:
\begin{align*}
 \frac{1}{(2\pi)^NN!}\prod_{1\le j< k \le N}\left|e^{\textnormal{i} \theta_j}-e^{\textnormal{i} \theta_k}\right|^2 d\theta_1 \cdots d\theta_N.
\end{align*}
This measure is also called the Circular Unitary Ensemble (CUE). There is a two-parameter generalization (the CUE is the special case $\beta=2$, $\delta=0$) of this measure, called the Circular Jacobi $\beta$-Ensemble, depending on parameters $\beta>0$ and $\Re(\delta)>-\frac{1}{2}$. This is the most natural extension of CUE for which explicit formulae for various quantities and connections to integrable systems exist, see \cite{BourgadeCircular,ForresterBook,ForresterWitteNagoya}. Its special case $\delta=0$ is called the Circular $\beta$-Ensemble ($\mathsf{C}\beta\mathsf{E}_N$) and is arguably the most well-known example of a beta ensemble from random matrix theory, see for example \cite{ForresterBook}. The Circular Jacobi $\beta$-Ensemble is then the following probability measure on $[0,2\pi]^N$ that we denote by $\mathsf{C}\mathsf{J}\beta\mathsf{E}_{N,\delta}$:
\begin{align*}
  \mathsf{C}\mathsf{J}\beta\mathsf{E}_{N,\delta}(d\boldsymbol\theta)=\frac{1}{c_{N,\beta,\delta}}\prod_{1\le j< k \le N}\left|e^{\textnormal{i} \theta_j}-e^{\textnormal{i} \theta_k}\right|^\beta  \prod_{j=1}^N \left(1-e^{-\textnormal{i}\theta_j}\right)^\delta \left(1-e^{\textnormal{i}\theta_j}\right)^{\overline{\delta}} d\theta_1 \cdots d\theta_N, 
\end{align*}
where $c_{N,\beta,\delta}$ is chosen so that this is a probability measure, see \cite{ForresterBook,BourgadeCircular}.
In the distinguished cases $\beta=\{1,2,4\}, \delta=0$ this is the law of the eigenangles of a random matrix that can be constructed by a natural transformation from a Haar distributed unitary matrix, see \cite{ForresterBook}. The cases $\beta=1$ and $\beta=4$ (and $\delta=0$) are called the Circular Orthogonal (COE) and Circular Symplectic Ensembles (CSE) respectively.  Matrix models also exist for all values of $\beta>0$ and $\Re(\delta)>-\frac{1}{2}$, see \cite{KilipNenciu,BourgadeCircular}. Finally, observe that for $\delta\in \frac{\beta}{2}\mathbb{N}$, $\mathsf{C}\mathsf{J}\beta\mathsf{E}_{N,\delta}$ coincides with $\mathsf{C}\beta\mathsf{E}_N$ conditioned to have eigenvalues at $1$ (equivalently eigenangles at $0$). 

Associated to $\boldsymbol{\theta}=(\theta_1,\dots,\theta_N)\in [0,2\pi]^N$ we define $\Phi_{\boldsymbol{\theta}}$ and $\Psi_{\boldsymbol{\theta}}$ as in (\ref{CharPolyDef}) and denote the expectation with respect to $\mathsf{C}\mathsf{J}\beta\mathsf{E}_{N,\delta}$ by $\mathbf{E}_{N,\beta,\delta}$. It is then natural to consider the joint moments corresponding to the $\mathsf{C}\mathsf{J}\beta\mathsf{E}_{N,\delta}$:
\begin{equation}
    F_{N,\beta,\delta}(s,h) := \mathbf{E}_{N,\beta,\delta}\left[ \left|\Psi_{\boldsymbol \theta}(0)\right|^{2s-2h} \left|\frac{d}{du}\Psi_{\boldsymbol \theta}(u)\bigg|_{u=0}\right|^{2h}\right].
\end{equation}
We note that these moments exist whenever $\beta>0$,  $\Re(\delta)>-\frac{1}{2}$ and $-\frac{1}{2}< h < \Re(\delta)+s+\frac{1}{2}$. The problem of the large $N$ asymptotics of $F_{N,\beta,0}(s,h)$ for $\beta\neq 2$, namely for $\mathsf{C}\beta\mathsf{E}_N$, was first considered in a recent paper by Forrester \cite{forrester2020joint}. As we will indicate at several places in the sequel, when one leaves the world of random unitary matrices ($\beta=2$, $\delta=0$) several structures break down. Nevertheless, the author in \cite{forrester2020joint} was able, using explicit computations with generalised hypergeometric functions, to prove convergence of the rescaled joint moments $F_{N,\beta,0}(s,h)$ and obtain an explicit combinatorial expression for the limit (that we recall in Theorem \ref{thm:forresterformula} below) for integer $s$ and real\footnote{In fact, as long as $s$ is an integer also complex $h$ can be considered, see \cite{forrester2020joint}.} $h$.

In this paper, using a different approach, exploiting results on consistent probability measures on interlacing arrays, we prove convergence for general complex $\delta \neq 0$ and positive real exponents $s$ and $h$ and give a probabilistic representation for the limit. Using the techniques presented in the sequel we also extend\footnote{We note that we do not require any new explicit computation to do this.} \footnote{In a revised version of the manuscript \cite{forrester2020joint} that was posted on the arXiv a week before the present paper appeared Forrester also obtains an explicit formula for $\delta=0$ (namely for the $\mathsf{C}\beta\mathsf{E}_N$, but this can be extended to real $\delta$) and real $s$ and integer $h$. The two proofs were obtained independently and are different from each other, see Remark \ref{ForresterRemark} for more details.}  Forrester's explicit formula for the limit to real $s$ and $\delta$ and integer $h$. Our more general goal in this paper is to provide a framework for the study of moments of the sum of points in rows of consistent random interlacing arrays, see Section \ref{AbstractConvSection}. Then, we apply (by virtue of Proposition \ref{ForresterProp}) this general theory to the problem of the joint moments of $\mathsf{C}\mathsf{J}\beta\mathsf{E}_{N,\delta}$ characteristic polynomials to obtain our main result, Theorem \ref{MainResult} below. Moreover, we also find an application in the study of the moments of the logarithmic derivative of the characteristic polynomial of the Laguerre $\beta$-ensemble, see Section \ref{LaguerreSection}.

To state our main result we require some notation. In what follows, $\mathbb{Y}$ denotes the set of all integer partitions (Young diagrams). Given $\kappa = (\kappa_1,\kappa_2,\kappa_3,\ldots) \in \mathbb{Y}$, we write $|\kappa| := \kappa_1+\kappa_2 +\kappa_3+ \cdots$; recall also that $(x)_k:=\prod_{j=0}^{k-1}(x+j)$ is the Pochhammer symbol, and $[x]_\kappa^{(\alpha)} := \prod_j \left(x - \frac{1}{\alpha}(j-1)\right)_{\kappa_j}$ is the generalised Pochhammer symbol associated to the partition $\kappa$. Given a box $\Box \in \kappa$, we let $\alpha(\Box)$ denote the arm length (the number of boxes to the right of $\Box$) and $\ell(\Box)$ denote the leg length (the number of boxes below $\Box$). We also define the co-arm length $\alpha^\prime(\Box)$ as the number of boxes to the left of $\Box$, and the co-leg length $\ell^\prime(\Box)$ as the number of boxes above $\Box$. For example, 
\begin{equation*}
\kappa = (4,2,1)\text{ corresponds to the Young diagram }\Yvcentermath1 \Yboxdim{8pt} \yng(4,2,1)\text{, with }|\kappa|=7\text{,} 
\end{equation*}
and the arm lengths, leg-lengths, co-arm lengths and co-leg lengths are given as follows:
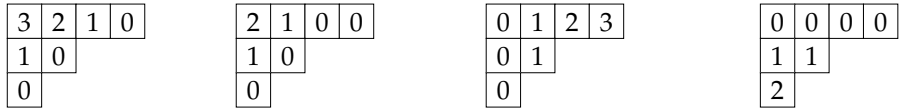
\begin{center}
\Yvcentermath0\Yboxdim{13pt}
\begin{tabular}{c  c  c  c}
     $ \young(3210,10,0)$& $\young(2100,10,0)$ & $\young(0123,01,0)$ & $\young(0000,11,2)$\\
     & & & \\
     Arm lengths $\alpha(\Box)$& Leg lengths $\ell(\Box)$ & Co-arm lengths $\alpha^\prime(\Box)$ & Co-leg lengths $\ell^\prime(\Box)$.
\end{tabular}
\end{center}
Following \cite{dalborgo2019}, we also introduce the function:
\begin{multline}
    \Upsilon_\beta(z) := \frac{\beta}{2}\log G\left(1+\frac{2z}{\beta}\right)-\left(z-\frac{1}{2}\right)\log \Gamma\left(1+\frac{2z}{\beta}\right) \\ +\int_0^\infty\left(\frac{1}{2x}-\frac{1}{x^2}+\frac{1}{x(e^x-1)}\right)\frac{e^{-xz}-1}{e^{x\beta/2}-1}dx + \frac{z^2}{\beta}+\frac{z}{2},
\end{multline}
where $G(z)$ denotes the Barnes $G$-function. We note that for special values of $\beta$ and $z$, $\Upsilon_\beta(z)$ takes an even more explicit form, see for example Lemma 7.1 in \cite{dalborgo2019}. Finally, the standard notation $\mathbb{E}\left[\mathsf{Z}\right]$ denotes the expectation of a random variable $\mathsf{Z}$ (unless we have introduced specific notation for the underlying probability measure in which case we use that instead). Our main result is:
\begin{thm}\label{MainResult}\label{prop:realbeta}
Let $\beta>0$. Let $s,h \in \mathbb{R}$ and $\delta \in \mathbb{C}$ be such that $\Re(\delta)>-\frac{1}{3}$, $s>-\frac{1}{3}$, $s+\Re(\delta)>0$ and $0\le h <s+\Re(\delta)+\frac{1}{2}$. Then, there exists a family of real random variables $\left\{\mathsf{X}_{\beta}(\tau)\right\}_{\tau\in \mathbb{C},\Re(\tau)>0}$ such that the following limit exists:
\begin{align} \label{eq:convergencethm}
    \lim\limits_{N \to \infty} \frac{1}{N^{2s^2/\beta + 2h}}F_{N,\beta,\delta}(s,h)\overset{\textnormal{def}}{=}{F}_{\beta,\delta}(s,h)={F}_{\beta,\delta}(s,0) 2^{-2h} \mathbb{E}\left[\left|\mathsf{X}_\beta(s+\delta)\right|^{2h}\right]<\infty,
\end{align}
where $F_{\beta,\delta}(s,0)$ is given explicitly as:
\begin{equation} \label{eq:upsilonformula}
F_{\beta,\delta}(s,0) = e^{2s\frac{\delta+\overline{\delta}}{\beta}+\Upsilon_\beta\left(1+\delta-\beta/2\right) - \Upsilon_\beta \left(1+\delta+\frac{s}{2}-\beta/2\right)+ \Upsilon_\beta\left(1+\overline{\delta}-\frac{\beta}{2}\right) - \Upsilon_\beta \left(1+\delta+\overline{\delta}-\frac{\beta}{2}\right) - \Upsilon_\beta \left(1+\overline{\delta}+\frac{s}{2}-\beta/2\right)+ \Upsilon_\beta\left(1+\delta+\overline{\delta}+s-\beta/2\right)}.
\end{equation}
When $\tau$ is real $\mathsf{X}_{\beta}(\tau)$ is symmetric about the origin in law. Moreover, for $h \in \mathbb{N} \cup \{0\}$ and $\tau>h-\frac{1}{2}$ we have the explicit formula:
\begin{align} \label{eq:explicitmoments}
  \mathbb{E}\left[\mathsf{X}_\beta(\tau)^{2h}\right]= (-1)^h \sum_{|\kappa|\le 2h}\frac{(-2h)_{|\kappa|} 2^{|\kappa|}}{[4\tau/\beta]_\kappa^{(\beta/2)}}\prod_{\Box \in \kappa}\frac{\frac{\beta}{2} \alpha^\prime(\Box) + \tau - \ell^\prime(\Box)}{\left(\frac{\beta}{2}(\alpha(\Box)+1)+ \ell(\Box)\right)\left(\frac{\beta}{2}\alpha(\Box)+ \ell(\Box) + 1\right)}.
\end{align}
\end{thm}

\begin{cor}
For $\beta>0,\delta>-\frac{1}{3}$, $s>-\frac{1}{3}$, $s+\delta>\frac{1}{2}$ and $0\le h <s+\delta+\frac{1}{2}$ we have $F_{\beta,\delta}(s,h)>0$.
\end{cor}
\begin{proof}
For $\tau>\frac{1}{2}$, evaluating \eqref{eq:explicitmoments} at $h=1$ gives, after further algebraic simplification, the following formula for the $2$nd moment of $\mathsf{X}_\beta (\tau)$:
\begin{equation*}
    \mathbb{E}\left[\mathsf{X}_\beta(\tau)^{2}\right]=\frac{\beta}{(2\tau-1)(4\tau+\beta)},
\end{equation*}
which does not vanish for any $\beta>0, \tau>\frac{1}{2}$. In particular, for these values of $\tau$ and $\beta$ the random variable $\mathsf{X}_\beta(\tau)$ is not almost surely zero. Combining this with \eqref{eq:convergencethm} gives the desired result.
\end{proof}

We expect the convergence statement in Theorem \ref{MainResult} to extend to the full range of parameters for which the joint moments exist but this would require some new ideas to establish. In fact, for $\beta=2$ and $\delta=0$ it is possible to go slightly beyond the parameter range for $s$ in Theorem \ref{MainResult} above, see \cite{assiotis2020joint}. However, the proof in \cite{assiotis2020joint} uses in an essential way the underlying determinantal point process structure which is absent for $\beta\neq 2$. Finally, there are a few results that have been proven for $\mathsf{X}_2(\tau)$ in the literature, see Remarks \ref{RmkPrincipalValue}, \ref{RemarkExplicit} and \ref{RemarkIntegrable} for more details, but the proofs of all of these again rely heavily on the determinantal point process structure and so they do not easily generalise to $\mathsf{X}_\beta(\tau)$.
 
\begin{rmk}\label{ForresterRemark}
In a revised version of the manuscript \cite{forrester2020joint}, posted on arXiv a week before the present paper first appeared, Forrester proves an explicit formula which is essentially equivalent to (\ref{eq:explicitmoments}). The two proofs were obtained independently and are different. In the initial version of \cite{forrester2020joint} the explicit formula for $F_{\beta,0}(s,h)$  is only obtained for integer $s$. In the revised version of \cite{forrester2020joint} additional explicit computations are performed that give a combinatorial formula for $F_{N,\beta,0}(s,h)/F_{N,\beta,0}(s,0)$ for real $s$ and integer $h$, and then the large $N$ limit is taken. The main point of our proof of (\ref{eq:explicitmoments}) is that the extension to real $s$ can be done already in the limit, using only the explicit formula of Forrester for $F_{\beta,0}(s,h)$ for integer $s$ and $h$, and further explicit computations are not required. The approach that we take is not evident and we use some general results developed in Section \ref{AbstractConvSection} in order to do this.
\end{rmk}

\begin{rmk}
For $h=0$ and general $\beta,s>0,\Re(\delta)>-\frac{1}{3}$ (in this case $F_{N,\beta,\delta}(s,0)$ is completely explicit using the Selberg integral, see \cite{KeatingSnaith}) a proof of the asymptotics was first given, as far as we are aware, in \cite{dalborgo2019}, see Proposition \ref{thm:h=0convergent} below. The asymptotics in the case $\delta=0$ and $s\in \mathbb{N}\cup \{0\}$ go back earlier and were first considered in \cite{KeatingSnaith}.
\end{rmk}

\begin{rmk}\label{RmkPrincipalValue}
For $\beta=2$ and real $\tau$ the random variable $\mathsf{X}_2(\tau)$ has a representation in terms of the principal value sum of points of a determinantal point process as proven by Qiu \cite{Qiu}. For general $\beta>0$ we expect an analogous representation in terms of the principal value sum of the eigenvalues of a certain stochastic operator, see \cite{Valko_Virag,Li_Valko}. Although, as far as we are aware, this has not been worked out explicitly, it might be possible to obtain using the techniques of \cite{Valko_Virag,Li_Valko}.
\end{rmk}

\begin{rmk}\label{RemarkExplicit}
We expect that for general $\tau$, the random variable $\mathsf{X}_\beta(\tau)$ is not almost surely zero, so that in particular $F_{\beta,\delta}(s,h)>0$. Establishing this for the full range of parameters turns out to be tricky, even in the case $\beta=2$ and real $\tau$ which was proven in \cite{assiotis2020joint}. This was done using the representation mentioned in Remark \ref{RmkPrincipalValue} above. In fact, we expect a much stronger result: that the law of $\mathsf{X}_\beta(\tau)$ has a density with respect to the Lebesgue measure. Remarkably, when $\tau\in \mathbb{N}\cup \left\{0\right\}$ this density is completely explicit. This was first obtained in the case $\beta=2$ in \cite{Distinguished} and for general $\beta>0$, by a different method, in \cite{forrester2020joint}.
\end{rmk}

\begin{rmk}\label{RemarkIntegrable}
For $\beta=2$ and any real $\tau>-\frac{1}{2}$, the characteristic function $t\mapsto \mathbb{E}\left[e^{\textnormal{i}\frac{t}{2}\mathsf{X}_2(\tau)}\right]$ is a tau- function of a special case of the $\sigma$-Painlev\'e III' equation, which depends on the parameter $\tau$, see \cite{Distinguished} for more details. It is not clear whether such connections to integrable systems extend beyond $\beta=2$ (even the simpler determinantal case of general complex $\tau$ for $\beta=2$ is still open). It would be interesting to investigate this.
\end{rmk}

Let us say a word about the strategy of proof. Our starting point is an observation alluded to earlier, that was first made in \cite{assiotis2020joint}, and also used in \cite{Distinguished}, in the setting of $\beta=2$ and $\delta=0$, which connects $F_{N,\beta,\delta}(s,h)$ to the moments of the trace of the Hermitian Hua-Pickrell matrix ensemble (also known as Cauchy ensemble, see \cite{HuaBook,Pickrell,Neretin_2002,Borodin_Olshanski,ForresterBook,BourgadeCircular,ForresterWitte2000}). An analogous connection also exists for general $\beta>0$ and $\delta=0$, as established in \cite{forrester2020joint}. The further extension to $\delta\neq 0$ is straightforward and we present it in Section \ref{SectionPreliminaries}. Modulo this common start, our approach is completely different from the one of Forrester in \cite{forrester2020joint}. It is probabilistic in nature and makes heavy use of some hidden exchangeable structure, a feature shared with the approach in \cite{assiotis2020joint}. A key ingredient in \cite{assiotis2020joint} is the fact that one can correctly define a unitarily invariant Hua-Pickrell measure on infinite Hermitian matrices. Here instead we make use of some general results from \cite{assiotis2020boundary} on consistent (this will be made precise in the sequel) distributions, depending on a parameter $\beta$, on infinite interlacing arrays and further develop a little theory using some arguments based on exchangeability, see Section \ref{InterlacingConvSection}. Although the appearance of random infinite interlacing arrays might seem slightly unmotivated at first sight, this setting is, in some sense, the natural general $\beta$ analogue of random unitarily invariant infinite Hermitian matrices, see Section \ref{InterlacingConvSection} and Remark \ref{RmkMatrices} in particular for more details.


Finally, using the framework developed here we prove in Section \ref{LaguerreSection} an analogous result to Theorem \ref{MainResult} for the moments of the logarithmic derivative of the characteristic polynomial of the Laguerre $\beta$-ensemble, see Proposition \ref{LaguerreProp}. The random variables appearing in Proposition \ref{LaguerreProp} also appear in the asymptotics of the joint moments of characteristic polynomials from the classical compact groups. This is established, using different methods, in work in preparation by one us \cite{AlperDissertation}  and gives the conjectural asymptotics of joint moments over various L-function families, as the conductor of the family tends to infinity.

\paragraph{Acknowledgements} We are very grateful to two anonymous referees for a very careful reading of the paper and many useful comments and suggestions which have improved the presentation.

\section{Preliminaries}\label{SectionPreliminaries}

We need to introduce some notation and definitions and state some previous results. Define the Weyl chamber, for $N\in \mathbb{N}$, by:
\begin{align*}
    \mathbb{W}_N=\left\{\mathbf{x}=\left(x_1,\ldots,x_N\right)\in \mathbb{R}^N:x_1\ge \ldots \ge x_N \right\}.
\end{align*}
We will make frequent use, throughout the paper, without explicit mention, of the following fact. If $f$ is a symmetric function on $\mathbb{R}^N$ then,
$
 \int_{\mathbb{W}_N}f(\mathbf{x})d\mathbf{x}=\frac{1}{N!}\int_{\mathbb{R}^N}f(\mathbf{x})d\mathbf{x}.$

\begin{defn}\label{DefnHP}
For $\beta > 0$ and $\Re(\tau)>-\frac{1}{2}$, we introduce the probability measure on $\mathbb{W}_N$:
\begin{equation}
     \mathfrak{m}_{N,\beta}^{(\tau)}(d \mathbf{x}) := \frac{N!}{\mathcal{C}_{N,\beta}^{(\tau)}}\prod_{j=1}^N(1+\textnormal{i}x_j)^{-\tau-\beta(N-1)/2-1} (1-\textnormal{i}x_j)^{-\overline{\tau}-\beta(N-1)/2-1}\left|\Delta(\mathbf{x})\right|^\beta d\mathbf{x},
\end{equation}
where $\Delta(\mathbf{x}) := \prod_{1 \le i < j \le N} (x_i-x_j)$ is the Vandermonde determinant, and $\mathcal{C}_{N,\beta}^{(\tau)}$ is the normalisation constant, see \cite{ForresterBook}. We will need its explicit value only for real values of $\tau$, in which case it is given by, see \cite{ForresterBook}:
\begin{equation}
    \mathcal{C}_{N,\beta}^{(\tau)} = 2^{-\beta N(N-1)/2 -2N\tau}\pi^N \prod_{j=0}^{N-1}\frac{\Gamma\left(\frac{\beta}{2}j + 2\tau + 1\right)\Gamma\left(\frac{\beta}{2}(j+1)+1\right)}{\Gamma\left(\frac{\beta}{2}j+\tau+1\right)^2 \Gamma\left(\frac{\beta}{2}+1\right)}. \label{eq:normconst}
\end{equation}
\end{defn}

These are known as the general-$\beta$ Hua-Pickrell measures (also known as the Cauchy $\beta$-ensemble, see \cite{ForresterBook}). Throughout this paper we use $\mathbb{E}_{N,\beta}^{(\tau)}$ to denote expectations taken with respect to the measure $\mathfrak{m}_{N,\beta}^{(\tau)}$, and in the context of taking these expectations $(x_1^{(N)},\ldots,x_N^{(N)})$ denotes a point in $\mathbb{W}_N$ distributed according to $\mathfrak{m}_{N,\beta}^{(\tau)}$.  One of the key propositions in \cite{forrester2020joint} links the joint moments $F_{N,\beta,0}(s,h)$ to moments taken against the Hua-Pickrell measures. Following a similar method, we obtain the following generalization of \cite{forrester2020joint}[Proposition 2.1] to $\delta \neq 0$.

\begin{prop} For all $\beta > 0$, $\Re(\delta)>-\frac{1}{2}$, $s+\Re(\delta)>-\frac{1}{2}$ and $-\frac{1}{2} < h < \Re(\delta)+s + \frac{1}{2}$, we have:
\begin{equation}
    F_{N,\beta,\delta}(s,h) = F_{N,\beta,\delta}(s,0) 2^{-2h} \mathbb{E}_{N,\beta}^{(s+\delta)}\left[\left|x_1^{(N)}+\cdots+x_N^{(N)}\right|^{2h}\right].  
\end{equation}
 \label{ForresterProp}
 \begin{proof}
 We argue as in \cite{forrester2020joint}[Proposition 2.1]. Firstly, observe that:
 \begin{equation*}
     \frac{\Psi_{\boldsymbol \theta}'(0)}{\Psi_{\boldsymbol \theta}(0)}=-\frac{1}{2}\sum_{j=1}^N \cot\left(\frac{\theta_j}{2}\right).
 \end{equation*}
Thus, making the transformation $x_j=\cot\left(\frac{\theta_j}{2}\right)$ gives the equality:
 \begin{equation*}
     F_{N,\beta,\delta}(s,h)= K_{N,\beta,s,\delta} 2^{-2h} \mathbb{E}_{N,\beta}^{(s+\delta)}\left[\left|x_1^{(N)}+\cdots+x_N^{(N)}\right|^{2h}\right]
\end{equation*}
where $K_{N,\beta,s,\delta}$ is a constant that only depends on $N,\beta,s$ and $\delta$. Then, evaluating both sides at $h=0$ gives the statement of the proposition.  
 \end{proof}
\end{prop}

We will also use the following corollary of \cite[Lemma 4.14]{dalborgo2019}:
\begin{prop}\cite[Lemma 4.14]{dalborgo2019}\label{thm:h=0convergent}
Let $\beta>0$ and $\Re(\delta)>-\frac{1}{3}$. For $s > -\frac{1}{3}$, we have:
\begin{align*}
    \log F_{N, \beta,\delta }(s,0) =\frac{2s^2}{\beta}\log(N)+2s\frac{\delta+\overline{\delta}}{\beta}+\Upsilon_\beta\left(1+\delta-\beta/2\right) - \Upsilon_\beta \left(1+\delta+\frac{s}{2}-\beta/2\right)\\+ \Upsilon_\beta\left(1+\overline{\delta}-\frac{\beta}{2}\right) - \Upsilon_\beta \left(1+\delta+\overline{\delta}-\frac{\beta}{2}\right) - \Upsilon_\beta \left(1+\overline{\delta}+\frac{s}{2}-\beta/2\right)\\+ \Upsilon_\beta\left(1+\delta+\overline{\delta}+s-\beta/2\right)+ o(1),
\end{align*}
as $N \to \infty$.
\end{prop}
Thus, an immediate corollary of Proposition \ref{thm:h=0convergent} is the explicit evaluation \eqref{eq:upsilonformula}. 

Our final main ingredient is the following result of \cite{forrester2020joint}, giving an explicit evaluation of $F_{\beta,0}(s,h)$ for $s \in \mathbb{N} \cup \{0\}$, $h \in \mathbb{C}$, $-\frac{1}{2} < \Re(h) < s + \frac{1}{2}$. We will use this formula to derive \eqref{eq:explicitmoments}; that is, the explicit evaluation for $h \in \mathbb{N} \cup \{0\}$ and $\tau > h - \frac{1}{2}$. 

\begin{thm}\cite[Theorem 1.1]{forrester2020joint}\label{thm:forresterformula}
For $\beta > 0$, $s \in \mathbb{N} \cup \{0\},$ and $-\frac{1}{2} < \Re(h) < s + \frac{1}{2}$, we have the explicit evaluation:
\begin{multline}
        F_{\beta,0}(s,h) = \prod_{j=1}^s \frac{\Gamma(2j/\beta)}{\Gamma(2(s+j)/\beta)} \times \frac{1}{2^{2h} \cos(\pi h)} \\ \times \sum_{\kappa}\frac{(-2h)_{|\kappa|} 2^{|\kappa|}}{[4s/\beta]_\kappa^{(\beta/2)}}\prod_{\Box \in \kappa}\frac{\frac{\beta}{2} \alpha^\prime(\Box) + s - \ell^\prime(\Box)}{\left(\frac{\beta}{2}(\alpha(\Box)+1)+ \ell(\Box)\right)\left(\frac{\beta}{2}\alpha(\Box)+ \ell(\Box) + 1\right)}. \label{eq:forresterformula}
\end{multline}
Here the sum over $\kappa$ ranges over all partitions with at most $s$ parts.
\end{thm}

\section{Proof of the convergence statement in Theorem \ref{MainResult}}\label{AbstractConvSection}
\subsection{A moments convergence result for rows of random interlacing arrays}\label{InterlacingConvSection}

In this section we prove some results on the moments of averages of the points in rows of consistent (this will be made precise shortly) random infinite interlacing arrays. As we explain in Remark \ref{RemarkKernel} and Remark \ref{RmkMatrices} below such random arrays, for $\beta=1,2,4$, arise by looking at the eigenvalues of consecutive principal submatrices of an infinite conjugation-invariant random matrix. The definition of a consistent random infinite interlacing array makes sense for all $\beta>0$ however. The main observation behind this section is that instead of studying the sum of eigenvalues of the matrix (namely the sum of points in rows of the array), for $\beta=1,2,4$, it is easier to study the sum of diagonal entries of the matrix  (of course these two sums are equal) because of their exchangeable structure. There is a natural generalisation of the notion of diagonal entries of a matrix to that of diagonal entries corresponding to an infinite interlacing array for any $\beta>0$. These still have an exchangeable structure and this is what makes everything work. A key input to our approach are the results from \cite{assiotis2020boundary} which generalise the theory developed for the case $\beta=2$ in \cite{Olshanski_Vershik} and \cite{Borodin_Olshanski}. The arguments in this section are based on exchangeability and this seems to be the most general setting to which they apply in a random matrix context\footnote{There are many discrete analogues, see for example \cite{GTgraph,Petrov}, of the construction in \cite{assiotis2020boundary}, namely consistent (in a certain sense) random discrete infinite interlacing arrays, and it is plausible that a variation of these arguments can be used there as well. We do not pursue it further.}.

We need some notation and definitions. For $\mathbf{x}\in \mathbb{W}_N$ and $\mathbf{y}\in \mathbb{W}_{N+1}$, we write $\mathbf{x}\prec \mathbf{y}$ if the following inequalities hold:
\begin{align*}
 y_1\ge x_1 \ge y_2 \ge \cdots \ge y_N \ge x_N \ge y_{N+1}.   
\end{align*}

For any $\beta>0$ and $N\ge 1$, consider the following Markov kernel $\mathsf{\Lambda}^{(\beta)}_{N+1,N}$ from $\mathbb{W}_{N+1}$ to $\mathbb{W}_N$ given by the explicit formula, for $\mathbf{y}\in \mathbb{W}^{\circ}_{N+1}$ (the interior of $\mathbb{W}_{N+1}$):
\begin{align*}
  \mathsf{\Lambda}^{(\beta)}_{N+1,N}\left(\mathbf{y},d\mathbf{x}\right) =\frac{\Gamma\left(\frac{\beta}{2}(N+1)\right)}{\Gamma\left(\frac{\beta}{2}\right)^{N+1}}\prod_{1\le i<j \le N+1} (y_i-y_j)^{1-\beta}\prod_{1\le i<j\le N} (x_i-x_j)\prod_{i=1}^N \prod_{j=1}^{N+1}|x_i-y_j|^{\frac{\beta}{2}-1}\mathbf{1}_{\mathbf{x}\prec \mathbf{y}}d\mathbf{x},
\end{align*}
where $\mathbf{1}_{\mathcal{A}}$ is the indicator function of the set $\mathcal{A}$. The fact that this correctly integrates to one is a result of Dixon and Anderson, see \cite{Dixon,Anderson}. We note that this formula extends continuously to arbitrary $\mathbf{y}\in \mathbb{W}_{N+1}$. This is evident from an equivalent definition of the kernel $\mathsf{\Lambda}^{(\beta)}_{N+1,N}$ in terms of the distribution of the roots of a certain random polynomial associated to $\mathbf{y}\in \mathbb{W}_{N+1}$, see Definition 1.1 and Proposition 1.2 in \cite{assiotis2020boundary}.

\begin{rmk}\label{RemarkKernel}
This Markov kernel might seem to come out of thin air but for the values $\beta=1,2,4$ (which provide the motivation) $\mathsf{\Lambda}^{(\beta)}_{N+1,N}\left(\mathbf{y},\cdot\right)$ corresponds to the conditional distribution of the eigenvalues of the principal $N\times N$ submatrix of a random conjugation\footnote{Orthogonal ($\beta=1$) or unitary ($\beta=2$) or symplectic ($\beta=4$) conjugation respectively.}-invariant $(N+1)\times (N+1)$ self-adjoint matrix, with either real ($\beta=1$) or complex ($\beta=2$) or quaternion ($\beta=4$) entries, with given spectrum $\mathbf{y}=\left(y_1,\dots,y_{N+1}\right)$, see \cite{assiotis2020boundary} and Remark \ref{RmkMatrices} below.
\end{rmk}

The following definition will be important in what follows.

\begin{defn}
Let $\beta>0$ and $N\ge 1$. A  consistent random family of interlacing arrays with parameter $\beta$ and length $N$ is a family of random sequences $\left(\mathbf{x}^{(i)}\right)_{i=1}^N$ such that $\mathbf{x}^{(i)}\in \mathbb{W}_i$ with:
\begin{align*}
  \mathbf{x}^{(1)}\prec \mathbf{x}^{(2)} \prec \cdots \prec \mathbf{x}^{(N-1)} \prec \mathbf{x}^{(N)} 
\end{align*}
and the joint distribution $\mathcal{N}$ of the sequence $\left(\mathbf{x}^{(i)}\right)_{i=1}^N$ satisfies:
\begin{align*}
   \mathcal{N}\left(d\mathbf{x}^{(1)},\ldots,d\mathbf{x}^{(N)}\right)=\nu_N(d\mathbf{x}^{(N)})\mathsf{\Lambda}_{N,N-1}^{(\beta)}\left(\mathbf{x}^{(N)},d\mathbf{x}^{(N-1)}\right)\cdots \mathsf{\Lambda}^{(\beta)}_{2,1}\left(\mathbf{x}^{(2)},d\mathbf{x}^{(1)}\right),
\end{align*}
where $\nu_N$ is the distribution of the top row $\mathbf{x}^{(N)}$ in the sequence. A consistent random family of infinite interlacing arrays with parameter $\beta$ is a family of random infinite sequences $\left(\mathbf{x}^{(i)}\right)_{i=1}^{\infty}$ such that the above holds for all $N\ge 1$.
\end{defn}

\begin{rmk}\label{RmkMatrices} Consistent random families of interlacing arrays are very closely related to random conjugation-invariant matrices for $\beta=1, 2, 4$. Let $\mathbf{H}$ be a random infinite self-adjoint matrix with real (for $\beta=1$), complex (for $\beta=2$) or quaternion (for $\beta=4$) entries, so that the law of all its finite top-left submatrices is invariant under orthogonal (for $\beta=1$), unitary (for $\beta=2$) or symplectic (for $\beta=4$) conjugation. Then, the eigenvalues of the 
consecutive top-left submatrices of $\mathbf{H}$ form a consistent family of infinite interlacing arrays for parameter $\beta$. Conversely any consistent random family of infinite interlacing arrays for $\beta\in\{ 1, 2, 4\}$ comes from such an infinite conjugation-invariant random matrix $\mathbf{H}$, see Proposition 1.7 in \cite{assiotis2020boundary} for a proof.
\end{rmk}

We now define the so-called diagonal entries $\left(\mathsf{d}_1,\mathsf{d}_2,\mathsf{d}_3,\ldots\right)$ of an interlacing array $\left(\mathbf{x}^{(i)}\right)_{i=1}^{\infty}$ by $\mathsf{d}_1=\mathbf{x}^{(1)}$ and for $i\ge 1$:
\begin{equation*}
    \mathsf{d}_{i+1}=\sum_{j=1}^{i+1}x_j^{(i+1)}-\sum_{j=1}^{i}x_j^{(i)}.
\end{equation*}

\begin{rmk}
This terminology comes from the fact that it coincides with the usual notion of diagonal entries of matrices for $\beta \in \{1, 2, 4 \}$, see Remark \ref{RmkMatrices}. 
\end{rmk}

We have the following result for the diagonal entries of consistent random infinite interlacing arrays. 

\begin{prop}\label{DiagExchProp}
Let $\beta>0$. Let $\mathfrak{M}$ be the law of a consistent random infinite interlacing array with parameter $\beta$. Then, the random infinite sequence of diagonal entries $\left(\mathsf{d}_1,\mathsf{d}_2,\mathsf{d}_3,\ldots\right)$ associated to the array with law $\mathfrak{M}$ is exchangeable.
\end{prop}

\begin{proof} Fix $\beta>0$ and let $N\ge 1$ be arbitrary. For any consistent distribution/law $\mathfrak{M}$ on infinite interlacing arrays (with parameter $\beta$) we write
$\mathsf{Law}_{\mathfrak{M}}\left(\mathsf{d}_1,\dots,\mathsf{d}_N\right)$ for the joint law of the diagonal elements
$\mathsf{d}_1,\dots,\mathsf{d}_N$. The extremal consistent distributions on infinite interlacing arrays (namely the ones that cannot be written as a convex combination of other consistent distributions) have been classified in \cite{assiotis2020boundary}. They are parametrised by an infinite-dimensional space $\Omega$ endowed with a certain topology, see Definition 1.9 in \cite{assiotis2020boundary}. For any $\omega\in \Omega$ we denote by $\mathfrak{N}_{\omega}^{(\beta)}$ the corresponding extremal consistent distribution. Then, under $\mathfrak{N}_\omega^{(\beta)}$ the diagonal elements are i.i.d., see Theorem 1.13 in \cite{assiotis2020boundary}, namely
\begin{equation*}
 \mathsf{Law}_{\mathfrak{N}_\omega^{(\beta)}}\left(\mathsf{d}_1,\dots,\mathsf{d}_N\right)=\mathfrak{n}_{\omega}^{(\beta)}(dx_1)\cdots \mathfrak{n}_{\omega}^{(\beta)}(dx_N),
\end{equation*}
where the probability measure $\mathfrak{n}_\omega^{(\beta)}$ on $\mathbb{R}$ is explicit, see Theorem 1.13 in \cite{assiotis2020boundary}. Moreover, from Theorem 1.16 in \cite{assiotis2020boundary} we obtain that for any consistent distribution $\mathfrak{M}$ there exists a unique Borel probability measure $\nu^{\mathfrak{M}}$ on $\Omega$ such that
\begin{equation*}
 \mathsf{Law}_{\mathfrak{M}}\left(\mathsf{d}_1,\dots,\mathsf{d}_N\right)=\int_{\Omega}\nu^{\mathfrak{M}}(d\omega)\mathsf{Law}_{\mathfrak{N}_\omega^{(\beta)}}\left(\mathsf{d}_1,\dots,\mathsf{d}_N\right)=\int_{\Omega}\nu^{\mathfrak{M}}(d\omega)\mathfrak{n}_{\omega}^{(\beta)}(dx_1)\cdots \mathfrak{n}_{\omega}^{(\beta)}(dx_N).
\end{equation*}
Now, let $\sigma$ be an arbitrary permutation from the symmetric group on $N$ symbols. Then,
\begin{equation*}
 \mathsf{Law}_{\mathfrak{M}}\left(\mathsf{d}_{\sigma(1)},\dots,\mathsf{d}_{\sigma(N)}\right)=\int_{\Omega}\nu^{\mathfrak{M}}(d\omega)\mathfrak{n}_{\omega}^{(\beta)}(dx_{\sigma(1)})\cdots \mathfrak{n}_{\omega}^{(\beta)}(dx_{\sigma(N)})= \mathsf{Law}_{\mathfrak{M}}\left(\mathsf{d}_1,\dots,\mathsf{d}_N\right)
\end{equation*}
and since $N\ge 1$ is arbitrary the desired conclusion follows.
\end{proof}

We also need the following notion of consistent sequences of probability measures on the Weyl chambers. Intuitively it can be understood as looking at the individual rows of a consistent random infinite interlacing array.

\begin{defn}
We say that a sequence of probability measures $\{\mu_N\}_{N=1}^{\infty}$ on $\{\mathbb{W}_N \}_{N\ge 1}$ is consistent for a parameter $\beta$ if:
\begin{align}\label{consistency}
 \mu_{N+1}\mathsf{\Lambda}_{N+1,N}^{(\beta)}=\mu_N, \ \ \forall N\ge 1.   
\end{align}
\end{defn}

Using Kolmogorov's extension theorem we see that consistent distributions of parameter $\beta$ on infinite interlacing arrays are in canonical bijection (under which the law of the $N$-th row of the array is given by $\mu_N$) with consistent sequences of probability measures for parameter $\beta$ defined above. 

We have the following general convergence result from \cite{assiotis2020boundary} for the sample mean of the elements on the $N$-th row.

\begin{prop}\label{TraceConv}
Let $\beta>0$. Let $\{\mu_N \}_{N=1}^{\infty}$ be a consistent sequence of probability measures for the parameter $\beta$. Let $\mathfrak{M}$ be the law of the corresponding consistent random infinite interlacing array. Denote by $\left(\mathsf{x}^{(N)}_1,\ldots, \mathsf{x}^{(N)}_N\right)$ the elements of the $N$-th row (having law $\mu_N$) and consider the average
$\mathsf{T}_N=N^{-1}\left(\mathsf{x}_1^{(N)}+\cdots+\mathsf{x}_N^{(N)}\right)$. Then, $\mathsf{T}_N$ converges $\mathfrak{M}$-almost surely as $N\to \infty $ to some random variable $\mathsf{T}_\infty$.
\end{prop}

\begin{proof}
This is a consequence of Theorem 3.6 in \cite{assiotis2020boundary}. The O-V condition, see Definition 2.1 and Definition 2.2 in \cite{assiotis2020boundary}, for the parameter $\gamma_1^{(N)}$ therein, which holds almost surely, for any consistent distribution $\mathfrak{M}$ with parameter $\beta$ on infinite interlacing arrays, is exactly the convergence of $\mathsf{T}_N$ (as $\gamma_1^{(N)}=\mathsf{T}_N$).
\end{proof}

We now move on to upgrading this convergence to convergence of the moments. When $\mathsf{T}_1=\mathsf{x}_1^{(1)}=\mathsf{d}_1$ is an $L^1$ random variable this can be achieved in a rather neat way. One could in principle use Proposition \ref{TraceConv} above along with proving uniform integrability which can be done using exchangeability. This approach was taken in \cite{assiotis2020joint}. Here instead we give a somewhat different argument. Below we denote by $\mathbb{E}_{\mu}$ the expectation with respect to a probability measure $\mu$. 

\begin{prop}\label{TraceMomentsConv}
In the setting of Proposition \ref{TraceConv} above, suppose $r\ge1$ and $\mathbb{E}_{\mu_1}\left[\left|\mathsf{T}_1\right|^r\right]<\infty$. Then, for any $0\le t\le r$ we have
\begin{align*}
 \mathbb{E}_{\mu_N}\left[\left|\mathsf{T}_N\right|^t\right]  \overset{N\to \infty}{\longrightarrow} \mathbb{E}_{\mathfrak{M}}\left[\left|\mathsf{T}_\infty\right|^t\right]<\infty.
\end{align*}
\end{prop}

\begin{proof}
Clearly, we have $\mathsf{T}_N=N^{-1}\left(\mathsf{d}_1+\cdots+\mathsf{d}_N\right)$. It is then standard that the sequence $\left(\mathsf{T}_{-N}\right)_{N\le -1}=\left(\mathsf{T}_N\right)_{N\ge 1}$ forms a backward martingale with respect to the exchangeable filtration\footnote{For each $N\ge 1$, the sigma algebra $\mathcal{E}_{-N}$ consists of events $\mathcal{A}$ whose indicator functions $\mathbf{1}_{\mathcal{A}}$ are Borel functions of the random variables $\left(\mathsf{Y}_1,\mathsf{Y}_2,\mathsf{Y}_3,\ldots\right)$ such that for any permutation $\sigma$ of $\{1,\ldots,N\}$ we have: 
\begin{align*}
\mathbf{1}_{\mathcal{A}}\left(\mathsf{Y}_1,\mathsf{Y}_2,\mathsf{Y}_3,\dots\right)=\mathbf{1}_{\mathcal{A}}\left(\mathsf{Y}_{\sigma(1)},\mathsf{Y}_{\sigma(2)},\dots,\mathsf{Y}_{\sigma(N)},\mathsf{Y}_{N+1},\ldots\right).
\end{align*}
The sigma algebra $\mathcal{E}=\cap_{N\ge 1}\mathcal{E}_{-N}$ is called the exchangeable sigma algebra.}, see \cite{kallenberg2002foundations}. Then, since $r\ge 1$ and $\mathbb{E}_{\mu_1}\left[\left|\mathsf{T}_1\right|^r\right]<\infty$, from the backward martingale convergence theorem, see \cite{kallenberg2002foundations}, we have convergence\footnote{The backward martingale convergence theorem also gives almost sure convergence since $\mathbb{E}_{\mu_1}\left[\left|\mathsf{T}_1\right|\right]<\infty$ (this is the law of large numbers for exchangeable sequences). We note that Proposition \ref{TraceConv} on almost sure convergence also applies to consistent sequences of measures for which $\mathsf{T}_1$ is not in $L^1$.} in $L^r$:
\begin{align*}
 \mathbb{E}_{\mathfrak{M}}\left[\left|\mathsf{T}_N-\mathsf{T}_\infty\right|^{r}\right] \overset{N \to \infty}{\longrightarrow} 0. 
\end{align*}
The extension to $0\le t \le r$ follows from Jensen's inequality.
\end{proof}

Finally, we have the following explicit formula for the positive integer $r$-moments of the abstract random variable $\mathsf{T}_\infty$ in terms of the $r$-moments of the sums of elements of the rows $1$ up to $r$ of the array (which can be analysed in a concrete way in many cases). The proof relies in an essential way on exchangeability and is an abstraction to, what seems to be, the most general setting of an argument given in \cite{assiotis2020joint}.

\begin{prop}\label{FormulaProp}
In the setting of Proposition \ref{TraceConv}, let $r\in \mathbb{N}$ and assume $\mathbb{E}_{\mu_1}\left[\left|\mathsf{T}_1\right|^{r}\right]<\infty$. Then, we have the following formula:
\begin{align}\label{EvenMomentsFormula}
\mathbb{E}_{\mathfrak{M}}\left[\mathsf{T}_{\infty}^{r}\right] =\frac{1}{r!} \sum_{k=1}^{r} (-1)^{r-k} \binom{r}{k} \mathbb{E}_{\mu_k}\left[\left(\mathsf{x}_1^{(k)}+\cdots+\mathsf{x}^{(k)}_k\right)^{r}\right].
\end{align}
\end{prop}

\begin{proof} We prove the following more general statement. Let $\left(\mathsf{Y}_1,\mathsf{Y}_2,\mathsf{Y}_3,\ldots\right)$ be any sequence of exchangeable real random variables, with $\mathbb{E}\left[\left|\mathsf{Y}_1\right|^r\right]<\infty$, where $h\in \mathbb{N}$ and consider $\mathsf{Y}_\infty=\lim_{N\to \infty} N^{-1}\sum_{i=1}^N \mathsf{Y}_i$. This limit exists by the law of large numbers for exchangeable sequences (by assumption the $\mathsf{Y}_i$ are integrable), see \cite{kallenberg2002foundations}. Then, we have the following formula (by assumption all moments involved are finite):
\begin{align}\label{exchangeablemomentsformula}
 \mathbb{E}\left[\mathsf{Y}_{\infty}^{r}\right] =\frac{1}{r!} \sum_{k=1}^{r} (-1)^{r-k} \binom{r}{k} \mathbb{E}\left[\left(\mathsf{Y}_1+\cdots+\mathsf{Y}_k\right)^{r}\right],
\end{align}
which we will prove shortly. The statement of the proposition is then a consequence of (\ref{exchangeablemomentsformula}) and the simple observation:
\begin{align*}
\mathbb{E}_{\mathfrak{M}}\left[\left(\mathsf{d}_1+\cdots+\mathsf{d}_k\right)^{r}\right]=\mathbb{E}_{\mu_k}\left[\left(\mathsf{x}_1^{(k)}+\cdots+\mathsf{x}_k^{(k)}\right)^{r}\right].
\end{align*}
We now turn to the proof of (\ref{exchangeablemomentsformula}). We claim that
\begin{align}\label{2hMoments}
  \mathbb{E}\left[\mathsf{Y}_\infty^{r}\right]=\mathbb{E}\left[\mathsf{Y}_1\mathsf{Y}_2\cdots \mathsf{Y}_{r}\right].
\end{align}
Then, (\ref{exchangeablemomentsformula}) is a consequence of (\ref{2hMoments}) along with the formula
\begin{align*}
 \mathbb{E}\left[\mathsf{Y}_1\mathsf{Y}_2\cdots \mathsf{Y}_{r}\right]=\frac{1}{r!}\sum_{k=1}^{r}(-1)^{r-k}\binom{r}{k}\mathbb{E}\left[\left(\mathsf{Y}_1+\cdots+\mathsf{Y}_k\right)^{r}\right], 
\end{align*}
which follows from exchangeability and the elementary identity, see \cite{assiotis2020boundary} for a proof:
\begin{align*}
 w_1\cdots w_{r}=\frac{1}{r!}\sum_{k=1}^{r}(-1)^{r-k}\sum_{1\le i_1 <i_2<\cdots < i_k\le r}\left(w_{i_1}+w_{i_2}+\cdots + w_{i_{k}}\right)^{r}.
\end{align*}

It remains to prove (\ref{2hMoments}). We recall a couple of classical facts about infinite sequences of exchangeable random variables, see \cite{kallenberg2002foundations}. Let $\mathcal{E}$ denote the exchangeable sigma algebra. Then, by the de Finetti-Hewitt-Savage theorem the $\left(\mathsf{Y}_i\right)_{i=1}^{\infty}$ are conditionally i.i.d.\ given $\mathcal{E}$ and moreover (since the $\mathsf{Y}_i$ are integrable) by the backward martingale convergence theorem $\mathsf{Y}_\infty=\mathbb{E}\left[\mathsf{Y}_i|\mathcal{E}\right]$. Hence, we can compute using the tower property:
\begin{align*}
  \mathbb{E}\left[\mathsf{Y}_1\mathsf{Y}_2\cdots \mathsf{Y}_{r}\right]=\mathbb{E}\left[\mathbb{E}\left[\mathsf{Y}_1\cdots\mathsf{Y}_{r}|\mathcal{E}\right]\right]=\mathbb{E}\left[\mathbb{E}\left[\mathsf{Y}_1|\mathcal{E}\right]\cdots \mathbb{E}\left[\mathsf{Y}_{r}|\mathcal{E}\right]\right]=\mathbb{E}\left[\mathsf{Y}_\infty^{r}\right].
\end{align*}
\end{proof}
\subsection{Application to the Hua-Pickrell measures}

We now specialize the previous results to the case of the Hua-Pickrell measures from Definition \ref{DefnHP}. We begin with the following which says that the theory of Section \ref{InterlacingConvSection} is indeed applicable.

\begin{lem}\label{HPconsistency}
Let $\beta>0$ and $\Re(\tau)>-\frac{1}{2}$. The Hua-Pickrell measures $\left\{\mathfrak{m}_{N,\beta}^{(\tau)} \right\}_{N=1}^\infty$ are consistent with parameter $\beta$.
\end{lem}

\begin{proof}
The required multidimensional integral (\ref{consistency}) that needs to be checked follows from Lemma 2.2 of \cite{Neretin_Triangle}.
\end{proof}

We denote by $\mathsf{X}_\beta(\tau)$ the limiting random variable $\mathsf{T}_\infty$ from Proposition \ref{TraceConv} corresponding to $\left\{\mathfrak{m}_{N,\beta}^{(\tau)} \right\}_{N=1}^\infty$. We note that the law of $\mathsf{X}_\beta(\tau)$ is symmetric about the origin whenever $\tau \in \mathbb{R}$, as this holds for each pre-limit random variable $\mathsf{T}_N$ using the symmetry of $\mathfrak{m}_{N,\beta}^{(\tau)}$ for $\tau \in \mathbb{R}$.

\begin{prop}\label{HPmomentsconv}
Let $\beta,\Re(\tau)>0$ and $0\le h <\Re(\tau)+\frac{1}{2}$. Then, 
\begin{align*}
  \lim\limits_{N \to \infty} \frac{1}{N^{2h}}\mathbb{E}_{N,\beta}^{(\tau)}\left[\left|x_1^{(N)}+\cdots+x_N^{(N)}\right|^{2h}\right]  =\mathbb{E}\left[\left|\mathsf{X}_\beta(\tau)\right|^{2h}\right]<\infty.
\end{align*}
\end{prop}

\begin{proof}
This is an application of Proposition \ref{TraceMomentsConv} by virtue of Lemma \ref{HPconsistency} and the fact that
\begin{align*}
  \mathbb{E}_{1,\beta}^{(\tau)}\left[\left|x_1^{(1)}\right|^r\right]&=\frac{1}{\mathcal{C}_{1,\beta}^{(\tau)}}\int_{-\infty}^\infty\left|x\right|^r\left(1+\textnormal{i}x\right)^{-\tau-1}\left(1-\textnormal{i}x\right)^{-\bar{\tau}-1}dx\\
  &=\frac{1}{\mathcal{C}_{1,\beta}^{(\tau)}}\int_{-\infty}^\infty\left|x\right|^r(1+x^2)^{-\Re(\tau)-1}e^{2\Im(\tau)\arctan(x)}dx<\infty,
\end{align*}
for any $0\le r<1+2\Re(\tau)$, which can be chosen so that $r\ge 1$.
\end{proof}

\begin{proof}[Proof of the convergence statement in Theorem \ref{MainResult}]
This follows immediately by combining  Proposition \ref{ForresterProp}, Proposition \ref{thm:h=0convergent} along with Proposition \ref{HPmomentsconv} above.
\end{proof}

We finally record the following formula for the even moments of $\mathsf{X}_\beta(\tau)$ in terms of the moments of the Hua-Pickrell ensembles. This will be a key ingredient in proving the explicit formula (\ref{eq:explicitmoments}) by extending\footnote{One could in principle try to compute the moments appearing on the right hand side of (\ref{MomentsConnection}) directly. However, since we already know Theorem \ref{thm:forresterformula} there is no need for further explicit computations.} Forrester's result, Theorem \ref{thm:forresterformula}. Using this we can readily study the moments $\mathbb{E}\left[\mathsf{X}_\beta(\tau)^{2h}\right]$ as functions of $\beta$ and $\tau$ and obtain certain analytic properties and bounds. Proving these properties directly, using an approximation from the integral formula for finite $N$ appears difficult (one needs certain uniform in $N$ estimates which are not straightforward to obtain).

\begin{prop}\label{prop:sumoverfiniteN} Let $\beta > 0$, let $h \in \mathbb{N} \cup \{0\}$, and let $\Re(\tau) > h - \frac{1}{2}$. Then,
\begin{align}\label{MomentsConnection}
 \mathbb{E}\left[\mathsf{X}_\beta(\tau)^{2h}\right]=\frac{1}{(2h)!}\sum_{k=1}^{2h} (-1)^{2h-k} {2h \choose k}    \mathbb{E}_{k,\beta}^{(\tau)}\left[\left(x_1^{(k)}+\cdots+x_k^{(k)}\right)^{2h}\right].
\end{align}
\end{prop}

\begin{proof}
This is an application of Proposition \ref{FormulaProp} (the case $h=0$ is notation convention).
\end{proof}

\section{Proof of the explicit formula in Theorem \ref{MainResult}}
In this section we consider the case when $\tau \in \mathbb{R}$, and prove the explicit formula (\ref{eq:explicitmoments}) for the even moments of $\mathsf{X}_\beta(\tau)$ by making use of Theorem \ref{thm:forresterformula}, Proposition \ref{prop:sumoverfiniteN} and some estimates proven below. We begin by defining $\Tilde{\mathcal{C}}_{N,\beta}^{(\tau)}$, for $\Re(\tau)>-\frac{1}{2}$, by
\begin{equation}
    \Tilde{\mathcal{C}}_{N,\beta}^{(\tau)} = 2^{-\beta N(N-1)/2 -2N\tau}\pi^N \prod_{j=0}^{N-1}\frac{\Gamma\left(\frac{\beta}{2}j + 2\tau + 1\right)\Gamma\left(\frac{\beta}{2}(j+1)+1\right)}{\Gamma\left(\frac{\beta}{2}j+\tau+1\right)^2 \Gamma\left(\frac{\beta}{2}+1\right)} \label{eq:normconst2}.
\end{equation}
We note that for real $\tau$ this is precisely the normalization constant $\mathcal{C}_{N,\beta}^{(\tau)}$, from (\ref{eq:normconst}), for the Hua-Pickrell $\beta$-ensemble, but for complex $\tau$ these constants differ. In a similar spirit, for $\Re(\tau)>-\frac{1}{2}$, we define the linear operator $\mathfrak{E}_{N,\beta}^{(\tau)}$, acting on symmetric functions $f$ on $\mathbb{R}^N$, by
\begin{equation}
    \mathfrak{E}_{N,\beta}^{(\tau)}\left[f\right]=\frac{1}{\Tilde{\mathcal{C}}_{N,\beta}^{(s)}} \int_{\mathbb{R}^N} \frac{f(x_1,\ldots,x_N)\left|\Delta(\mathbf{x})\right|^\beta}{\prod_{j=1}^N (1+x_j^2)^{\beta(N-1)/2+1+\tau}}d\mathbf{x},
\end{equation}
whenever the integral exists. Similar to before, we note that $\mathfrak{E}_{N,\beta}^{(\tau)}\left[f\right]$ coincides with the expectation $\mathbb{E}^{(\tau)}_{N,\beta}[f(x_1^{(1)},\dots,x_N^{(N)})]$, taken with respect to the Hua-Pickrell  measure $\mathfrak{m}^{(\tau)}_{N,\beta}$ whenever $\tau$ is real, but they differ for general complex $\tau$.

\begin{lem}
Fix $h \in \mathbb{N} \cup \{0\}$. Whenever $\beta$ is an even integer:
\begin{equation}
    \tau \mapsto \mathfrak{E}_{N,\beta}^{(\tau)}\left[\left(x_1+\cdots+x_N\right)^{2h}\right]
\end{equation}
is a rational function. \label{lem:rationalins}
\end{lem}
\begin{proof}
We proceed as in \cite[Proposition 1.4]{assiotis2020joint}. Since $\beta$ is an even integer, we expand the factor $\left|\Delta(\mathbf{x})\right|^\beta$ as a polynomial in the variables $x_1, \ldots,x_N$. We find that $\mathfrak{E}_{N,\beta}^{(\tau)}\left[\left(x_1+\cdots+x_N\right)^{2h}\right]$ is expressible as a linear combination, with coefficients not dependent on $\tau$, of terms of the form:
\begin{equation}
    \frac{1}{\Tilde{\mathcal{C}}_{N,\beta}^{(\tau)}} \int_{\mathbb{R}^N} \prod_{j=1}^N \frac{x_j^{2m_j}}{(1+x_j^2)^{\beta(N-1)/2 + 1 + \tau}} d \mathbf{x}
\end{equation}
(No odd exponents appear, by symmetry about the origin). Recalling the standard evaluation
\begin{equation}
    \int_{-\infty}^\infty \frac{x^{2m}}{(1+x^2)^{\beta(N-1)/2 + 1 + \tau}} dx = \frac{\Gamma\left(m+\frac{1}{2}\right) \Gamma\left(\frac{\beta}{2}(N-1) + \tau-m+\frac{1}{2}\right)}{\Gamma\left(\frac{\beta}{2}(N-1) + 1 + \tau\right)},
\end{equation}
we are required to show that 
\begin{equation}
    2^{2{\tau} N}\prod_{j=1}^{N}\frac{\Gamma\left(\frac{\beta}{2}j + \tau+1\right)^2\Gamma\left(\frac{\beta}{2}(N-1) + \tau-m_j+\frac{1}{2}\right)}{\Gamma\left(\frac{\beta}{2}j+2\tau+1\right)\Gamma\left(\frac{\beta}{2}(N-1) + 1 + \tau\right)} \label{eq:rationalexprn}
\end{equation}
is a rational expression in $\tau$. Since $\beta$ is an even integer, we note that
\begin{equation}
    \frac{\Gamma\left(\frac{\beta}{2}j+\tau+1\right)}{\Gamma\left(\frac{\beta}{2}(N-1)+1+\tau\right)} = \left(\beta/2 j+\tau+1\right)^{-1}_{\beta/2(N-1-j)},
\end{equation}
which is evidently rational in $\tau$. Thus we are reduced to showing that
\begin{equation}
        2^{2\tau N}\prod_{j=1}^{N}\frac{\Gamma\left(\frac{\beta}{2}j + \tau+1\right)\Gamma\left(\frac{\beta}{2}(N-1) + \tau-m_j+\frac{1}{2}\right)}{\Gamma\left(\frac{\beta}{2}j+2\tau+1\right)} \label{eq:rationalexpr2}
\end{equation}
is rational in $\tau$. By Legendre's duplication formula, 
\begin{equation}
    \Gamma\left(\frac{\beta}{2}j+2\tau+1\right) = \text{const} \times 2^{2\tau} \Gamma\left(\tau+\frac{1}{2}+\frac{\beta}{4}j\right)\Gamma\left(\tau+1+\frac{\beta}{4}j\right),
\end{equation}
where the constant is independent of $\tau$. Therefore, \eqref{eq:rationalexpr2} equals the following, up to a constant factor independent of $\tau$:
\begin{equation}
    \prod_{j=1}^{N}\frac{\Gamma\left(\frac{\beta}{2}j + \tau+1\right)\Gamma\left(\frac{\beta}{2}(N-1) + \tau-m_j+\frac{1}{2}\right)}{\Gamma\left(\frac{\beta}{4}j+\tau+1\right)\Gamma\left(\frac{\beta}{4}j+\tau+\frac{1}{2}\right)}.
\end{equation}
Finally, we note that the expression
\begin{equation}
    \frac{\Gamma\left(\frac{\beta}{2}j + \tau+1\right)\Gamma\left(\frac{\beta}{2}(N-1) + \tau-m_j+\frac{1}{2}\right)}{\Gamma\left(\frac{\beta}{4}j+\tau+1\right)\Gamma\left(\frac{\beta}{4}j+\tau+\frac{1}{2}\right)}
\end{equation}
is rational in $\tau$ since, as $\beta$ is an even integer, one of $\frac{\beta}{4}j+1$ and $\frac{\beta}{4}j+\frac{1}{2}$ is an integer, and then the other is a half-integer.
\end{proof}
\begin{lem}\label{lem:betaestimates} Let $\tau>-\frac{1}{2}$ and $h \in \mathbb{N} \cup \{0\}$ with $h<\tau+\frac{1}{2}$.
The function given by
\begin{equation}
    g_N^{(\tau)}:\beta \mapsto \mathfrak{E}_{N,\beta}^{(\tau)}\left[\left(x_1+\cdots+x_N\right)^{2h}\right]
\end{equation}
is holomorphic in $\Re(\beta) \ge 0$, and belongs to the exponential class: there are constants $C>0$ and $\mu > 0$ such that $|g_N^{(\tau)}(\beta)| \le C e^{\mu|\beta|}$ whenever $\Re(\beta) \ge 0$. Moreover, it is bounded by a polynomial on the line $\Re(\beta)=0$.
\end{lem}
\begin{proof}
That $g_N^{(\tau)}(\beta)$ is analytic follows from a standard argument involving Fubini's and Morera's theorems. Using the bound $|\Delta(\mathbf{x})| \le N! \prod_{i=1}^N(1+x_i^2)^{(N-1)/2}$, see Lemma 3.1 in \cite{Winn_2012}, we obtain that:
\begin{equation}
\begin{aligned}
 \left| g_N^{(\tau)}(\beta)\right|&\le \frac{\left(N!\right)^{\Re(\beta)}}{\left|\Tilde{\mathcal{C}}_{N,\beta}^{(\tau)}\right|} \int_{\mathbb{R}^N} \frac{(x_1+\cdots+x_N)^{2h}}{\prod_{j=1}^N (1+x_j^2)^{1+\tau}}d\mathbf{x}
&\ll_{\tau,h,N}\frac{\left(N!\right)^{\Re(\beta)}}{\left|\Tilde{\mathcal{C}}_{N,\beta}^{(\tau)}\right|}.
\end{aligned}\label{eq:betabound}
\end{equation}
Using Stirling's formula in the form $\log\Gamma(z+1) = z\log(z) - z + \frac{1}{2}\log(z) + O(1)$, and that $\log(z+1) - \log(z) = 1/z + O(1/z^2)$, we see that as a function of $\beta$:  
\begin{multline}
\log\frac{\Gamma\left(\frac{\beta}{2}j+\tau+1\right)^2 \Gamma\left(\frac{\beta}{2}+1\right)}{\Gamma\left(\frac{\beta}{2}j + 2\tau + 1\right)\Gamma\left(\frac{\beta}{2}(j+1)+1\right)} \\ = \frac{1}{2}\log(\beta)-\frac{\beta}{2} \log(j) - \frac{\beta}{2}(j+1)\log\left(1+\frac{1}{j}\right) + O_{\tau,j}(1),
\end{multline}
in particular:
\begin{equation}
\left|\frac{\Gamma\left(\frac{\beta}{2}j+\tau+1\right)^2 \Gamma\left(\frac{\beta}{2}+1\right)}{\Gamma\left(\frac{\beta}{2}j + 2\tau + 1\right)\Gamma\left(\frac{\beta}{2}(j+1)+1\right)}\right| \ll_{\tau,j} |\beta|^{\frac{1}{2}}j^{-\Re(\beta)/2}e^{-\Re(\beta)/2},
\end{equation}
so that: 
\begin{equation}
\left(N!\right)^{\Re(\beta)}\left|\Tilde{\mathcal{C}}_{N,\beta}^{(\tau)} \right|^{-1} \ll_{\tau,N} 2^{\Re(\beta) N(N-1)/2}e^{-\Re(\beta)N/2}(N!)^{\Re(\beta)/2}|\beta|^{\frac{N}{2}}.
\end{equation}
Applying this to \eqref{eq:betabound} gives the required bounds for $\Re(\beta) > 0$ and $\Re(\beta) = 0$.
\end{proof}
We can now prove the explicit formula in Theorem \ref{prop:realbeta}.

\begin{proof}[Proof of the explicit formula in Theorem \ref{MainResult}] Combining Proposition \ref{ForresterProp} with the convergence result (Proposition \ref{HPmomentsconv}), and then using Proposition \ref{prop:sumoverfiniteN},   we obtain the following chain of equalities, valid for all $\beta > 0$, $h \in \mathbb{N} \cup \{0\}$, and $\tau \in \mathbb{R}$ with $\tau > h - \frac{1}{2}$: \begin{equation}
\begin{aligned}
F_{\beta,0}(\tau,h) &= F_{\beta,0}(\tau,0) 2^{-2h} \mathbb{E}\left[\mathsf{X}_\beta(\tau)^{2h}\right]\\
&= F_{\beta,0}(\tau,0) \cdot \frac{2^{-2h}}{(2h)!}\sum_{k=1}^{2h} (-1)^{2h-k} {2h \choose k}    \mathbb{E}_{k,\beta}^{(\tau)}\left[\left(x_1^{(k)}+\cdots+x_k^{(k)}\right)^{2h}\right].
\end{aligned} \label{eq:useresults}
\end{equation}
Since $F_{\beta,0}(\tau,0) = \prod_{j=1}^\tau \frac{\Gamma(2j/\beta)}{\Gamma(2(\tau+j)/\beta)}$ whenever $\tau \in \mathbb{N} \cup \{0\}$, (as follows from evaluation of \eqref{eq:forresterformula} at $h=0$), it follows from Theorem \ref{thm:forresterformula} that, whenever $\beta > 0$, $\tau \in \mathbb{N} \cup \{0\}$, and $h \in \mathbb{N}\cup\{0\}, \ h \le \tau$, we have the equality, after recalling that $\mathfrak{E}^{(\tau)}_{k,\beta}$ and $\mathbb{E}^{(\tau)}_{k,\beta}$ coincide for real $\tau$:
\begin{multline}
\frac{1}{(2h)!}\sum_{k=1}^{2h} (-1)^{2h-k} {2h \choose k}    \mathfrak{E}_{k,\beta}^{(\tau)}\left[\left(x_1^{(k)}+\cdots+x_k^{(k)}\right)^{2h}\right] \\= (-1)^h \sum_{|\kappa|\le 2h}\frac{(-2h)_{|\kappa|} 2^{|\kappa|}}{[4\tau/\beta]_\kappa^{(\beta/2)}}\prod_{\Box \in \kappa}\frac{\frac{\beta}{2} \alpha^\prime(\Box) + \tau - \ell^\prime(\Box)}{\left(\frac{\beta}{2}(\alpha(\Box)+1)+ \ell(\Box)\right)\left(\frac{\beta}{2}\alpha(\Box)+ \ell(\Box) + 1\right)}.
\label{eq:extendpars}
\end{multline}
Now fix $\beta$ as an even integer. By Lemma \ref{lem:rationalins}, the left-hand side of \eqref{eq:extendpars} is a rational function in $\tau$. We observe that the right-hand side of \eqref{eq:extendpars} is also rational in $\tau$. Therefore the equality \eqref{eq:extendpars} holds for all $\beta \in 2 \mathbb{N}$, $h \in \mathbb{N} \cup \{0\}$, and $\tau \in \mathbb{R}, \tau > h - \frac{1}{2}$. Now, fix $\tau \in \mathbb{R}, \tau > h - \frac{1}{2}$. By Lemma \ref{lem:betaestimates}, the left-hand side of \eqref{eq:extendpars} satisfies the conditions of Carlson's theorem as a function of $\beta$; the right-hand side of \eqref{eq:extendpars} also satisfies the conditions of Carlson's theorem, being a rational expression in $\beta$. Therefore, we conclude that \eqref{eq:extendpars} holds for all $\beta > 0$, $h \in \mathbb{N} \cup \{0\}$ and all $\tau > h - \frac{1}{2}$. Combining this with \eqref{eq:useresults}, and  recalling that $\mathfrak{E}^{(\tau)}_{k,\beta}$ and $\mathbb{E}^{(\tau)}_{k,\beta}$ coincide for real $\tau$, we get the desired explicit formula.
\end{proof}

\section{On the moments of the logarithmic derivative of the Laguerre $\beta$-ensemble characteristic polynomial}\label{LaguerreSection}

In this section we study the moments of the logarithmic derivative of the characteristic polynomial of the Laguerre $\beta$-ensemble, using the theory in Section \ref{InterlacingConvSection}, and prove an analogue of our main result. For $\mathbf{x}\in \mathbb{W}_N$ consider the following polynomial:
\begin{align*}
 P_{\mathbf{x}}(t)=\prod_{i=1}^N\left(t-x_i\right).
\end{align*}
For $\beta>0$ and $\nu>-1$ introduce the probability measure on $\mathbb{W}^+_N=\mathbb{W}_N\cap [0,\infty)^N$:
\begin{align}\label{LaguerreEnsemble}
    \mathfrak{L}_{N,\beta}^{(\nu)}(d\mathbf{x})= \frac{N!}{\mathfrak{l}_{N,\beta}^{(\nu)}}\prod_{j=1}^N x_j^\nu e^{-x_j} \left|\Delta(\mathbf{x})\right|^\beta\mathbf{1}_{\mathbf{x}\in \mathbb{W}^+_N} d\mathbf{x},
\end{align}
where the normalisation constant $\mathfrak{l}_{N,\beta}^{(\nu)}$, see \cite{ForresterBook}, is given by
\begin{equation}
\mathfrak{l}_{N,\beta}^{(\nu)} = \prod_{j=0}^{N-1} \frac{\Gamma\left(\nu + 1 + \frac{\beta}{2}j\right)\Gamma\left(1+\frac{\beta}{2}(j+1)\right)}{\Gamma\left(1+\frac{\beta}{2}\right)}.\nonumber
\end{equation}
This is called the Laguerre $\beta$-ensemble. For $\beta=1,2 ,4$ it corresponds to the law of eigenvalues of a certain self-adjoint random matrix with either real or complex or quaternion entries respectively, see \cite{ForresterBook}. Tridiagonal models for which (\ref{LaguerreEnsemble}) is the law of the eigenvalues exist for all values of $\beta>0$, see \cite{DumitriuEdelman}. Under the transformation $x\mapsto \frac{2}{x}$, $\mathfrak{L}_{N,\beta}^{(\nu)}$ transforms to the measure:
\begin{align*}
    \mathfrak{IL}_{N,\beta}^{(\nu)}(d\mathbf{x})= \frac{N! }{\mathfrak{l}_{N,\beta}^{(\nu)}}2^{N \nu + \beta N(N-1)/2 + N} \prod_{j=1}^N x_j^{-\nu-(N-1)\beta-2} e^{-2/x_j} \left|\Delta(\mathbf{x})\right|^\beta \mathbf{1}_{\mathbf{x}\in \mathbb{W}^+_N} d\mathbf{x}.
\end{align*}
We denote expectation with respect to $\mathfrak{L}_{N,\beta}^{(\nu)}$ by $\mathcal{E}_{N,\beta}^{(\nu)}$ and with respect to $\mathfrak{IL}_{N,\beta}^{(\nu)}(d\mathbf{x})$ by $\hat{\mathcal{E}}_{N,\beta}^{(\nu)}$. We then consider the moments of the logarithmic derivative of the Laguerre $\beta$-ensemble characteristic polynomial:
\begin{align*}
G_{N,\beta}(\nu,r)=\mathcal{E}_{N,\beta}^{(\nu)}\left[\left|\frac{d}{dt}\log P_{\mathbf{x}}(t)\bigg|_{t=0}\right|^{r}\right].
\end{align*}
We note that these moments are finite for $\beta>0$, $\nu>-1$ and $r<\nu+1$. We then have the following analogue of our main result:

\begin{prop}\label{LaguerreProp}
Let $\beta,\nu>0$ and $0\le r < \nu+1$. Then, there exists a family of non-negative random variables $\left\{\mathsf{Y}_\beta(\nu)\right\}_{\nu>0}$ such that
\begin{align*}
        \lim\limits_{N \to \infty} \frac{1}{N^{r}}G_{N,\beta}(\nu,r)= \mathbb{E}\left[\mathsf{Y}_\beta(\nu)^{r}\right]<\infty.
\end{align*}
Moreover, for $r\in \mathbb{N}\cup\{0\}$ and $\nu>r-1$ we have the explicit formula
\begin{align}
\mathbb{E}\left[\mathsf{Y}_\beta(\nu)^{r}\right]= \frac{r!}{\beta^r} \sum_{|\kappa| = r} \prod_{\Box \in \kappa}\frac{1}{\left(\frac{\beta}{2}(\alpha(\Box)+1)+ \ell(\Box)\right)\left(\frac{\beta}{2}\alpha(\Box)+ \ell(\Box) + 1\right)} \prod_{j=0}^{\ell(\kappa) - 1} \frac{\Gamma\left(\nu+ \frac{\beta}{2}j + 1 - \kappa_j\right)}{\Gamma\left(\nu + \frac{\beta}{2}j + 1\right)}. \label{eq:InvLaguerreExplicit}
\end{align}
\end{prop}

\begin{rmk}
It follows from the results in \cite{AssiotisWishart}, see also Section 3 in \cite{Distinguished}, that $\mathsf{Y}_2(\nu)$ is equal to the sum of the (random) inverse points of the Bessel determinantal point process, see \cite{ForresterBook}. Moreover, for general $\beta$ it follows from the results of \cite{RamirezRider2009} that $\mathsf{Y}_\beta(\nu)$ is the trace of the random integral operator in (1.4) in \cite{RamirezRider2009} (after multiplication by $\beta/2$ and the correspondence of parameters $\nu=\beta(a+1)/2-1$). This random operator is the inverse operator to the generator of a diffusion process with random scale function and speed measure, see (1.3) in \cite{RamirezRider2009}. It also follows that $\mathsf{Y}_\beta(\nu)$ is in fact almost surely strictly positive.
\end{rmk}

\begin{rmk}
For $\beta=2$ and any $\nu>-1$, the Laplace transform $t\mapsto \mathbb{E}\left[e^{-4t\mathsf{Y}_2(\nu)}\right]$ is a tau-function of a special case, which is different from the one appearing in the case of $\mathsf{X}_2(\tau)$, of the $\sigma$-Painlev\'{e} III' equation, depending on the parameter $\nu$, see Section 3 in  \cite{Distinguished}.
\end{rmk}

We apply the results of Section \ref{InterlacingConvSection} and begin with the following:
\begin{lem}\label{InverseLaguerreConsistency}
Let $\beta>0$ and $\nu>-1$. The invere Laguerre measures $\left\{\mathfrak{IL}_{N,\beta}^{(\nu)}\right\}_{N=1}^\infty$ are consistent with parameter $\beta$.
\end{lem}

\begin{proof}
The required multidimensional integral (\ref{consistency}) that needs to be checked follows from Variant A right after Lemma 2.2 of \cite{Neretin_Triangle}.
\end{proof}

The following formula is an easy consequence of the results of Forrester \cite{forrester2020joint}.
\begin{prop}
Let $\beta>0$, $r \in \mathbb{N}\cup \{0\}$ and $\nu>r-1$. Then, we have
\begin{multline}
\hat{\mathcal{E}}_{N,\beta}^{(\nu)}\left[\left(x_1^{(N)}+\cdots+x_N^{(N)}\right)^r\right] \\ = \frac{r!}{\beta^r} \sum_{|\kappa| = r} \prod_{\Box \in \kappa}\frac{\frac{\beta}{2}\alpha^\prime(\Box) + N - \ell^\prime(\Box)}{\left(\frac{\beta}{2}(\alpha(\Box)+1)+ \ell(\Box)\right)\left(\frac{\beta}{2}\alpha(\Box)+ \ell(\Box) + 1\right)} \prod_{j=0}^{\ell(\kappa) - 1} \frac{\Gamma\left(\nu+ \frac{\beta}{2}j + 1 - \kappa_j\right)}{\Gamma\left(\nu + \frac{\beta}{2}j + 1\right)}. \label{eq:finiteNLaguerremoment}
\end{multline}
\end{prop}

\begin{proof}
By Proposition 4.1 of \cite{forrester2020joint}, we have the exact evaluation (we have corrected here a small misprint from \cite{forrester2020joint})
\begin{multline}
\frac{1}{S_N(\nu,\mu,\beta)}\int_{[0,1]^N}\left(\sum_{i=1}^N\frac{1}{x_i}\right)^r \prod_{j=1}^Nx_j^\nu(1-x_j)^\mu \left|\Delta(\mathbf{x})\right|^\beta d \mathbf{x} \\
= \left(\frac{2}{\beta}\right)^r r!\sum_{|\kappa| = r} \prod_{\Box \in \kappa}\frac{\frac{\beta}{2} \alpha^\prime(\Box) + N - \ell^\prime(\Box)}{\left(\frac{\beta}{2}(\alpha(\Box)+1)+ \ell(\Box)\right)\left(\frac{\beta}{2}\alpha(\Box)+ \ell(\Box) + 1\right)} \\  \times \prod_{j=0}^{\ell(\kappa) - 1} \frac{\Gamma\left(\nu+ \frac{\beta}{2}j + 1 - \kappa_j\right)}{\Gamma\left(\nu + \frac{\beta}{2}j + 1\right)} \cdot \frac{\Gamma\left(\nu+ \mu + \frac{\beta}{2}(N+j-1) + 2 \right)}{\Gamma\left(\nu + \mu + \frac{\beta}{2}(N+j-1) + 2-\kappa_j \right)},\label{eq:forresterjacobi}
\end{multline}
where $S_N(\nu,\mu,\beta)$ is the Selberg normalisation:
\begin{equation*}
    S_N(\nu,\mu,\beta) = \prod_{j=0}^{N-1}\frac{\Gamma\left(\nu+1+\frac{\beta}{2}j\right)\Gamma\left(\mu+1+\frac{\beta}{2}j\right)\Gamma\left(1+\frac{\beta}{2}(j+1)\right)}{\Gamma\left(\nu+\mu+2+\frac{\beta}{2}(N+j-1)\right)\Gamma\left(1+\frac{\beta}{2}\right)}.
\end{equation*}
Now making the substitution $x_i \mapsto \mu x_i$ in \eqref{eq:forresterjacobi}, gives
\begin{multline}
    \frac{1}{S_N(\nu,\mu,\beta)}\frac{1}{\mu^{N+aN + \beta N(N-1)/2}}\int_{[0,\mu]^N}\left(\sum_{i=1}^N\frac{1}{x_i}\right)^r \prod_{j=1}^Nx_j^\nu\left(1-\frac{x_j}{\mu}\right)^\mu \left|\Delta(\mathbf{x})\right|^\beta d \mathbf{x} \\
=  \left(\frac{2}{\beta}\right)^r r!\sum_{|\kappa| = r} \prod_{\Box \in \kappa}\frac{\frac{\beta}{2} \alpha^\prime(\Box) + N - \ell^\prime(\Box)}{\left(\frac{\beta}{2}(\alpha(\Box)+1)+ \ell(\Box)\right)\left(\frac{\beta}{2}\alpha(\Box)+ \ell(\Box) + 1\right)} \\ \times \prod_{j=0}^{\ell(\kappa) - 1} \frac{\Gamma\left(\nu+ \frac{\beta}{2}j + 1 - \kappa_j\right)}{\Gamma\left(\nu + \frac{\beta}{2}j + 1\right)} \cdot \frac{\Gamma\left(\nu+ \mu + \frac{\beta}{2}(N+j-1) + 2\right)}{\mu^{\kappa_j}\Gamma\left(\nu + \mu + \frac{\beta}{2}(N+j-1) + 2- \kappa_j\right)}. \label{eq:forresterjacobi2}
\end{multline}
Hence, taking the limit $\mu \to \infty$ in \eqref{eq:forresterjacobi2} (using the dominated convergence theorem), yields
\begin{multline}
    \frac{1}{\mathfrak{l}_{N,\beta}^{\nu}}\int_{[0,\infty)^N}\left(\sum_{i=1}^N\frac{1}{x_i}\right)^r \prod_{j=1}^Nx_j^\nu e^{-x_j} \left|\Delta(\mathbf{x})\right|^\beta d \mathbf{x} \\
=  \left(\frac{2}{\beta}\right)^r r!\sum_{|\kappa| = r} \prod_{\Box \in \kappa}\frac{\frac{\beta}{2} \alpha^\prime(\Box) + N - \ell^\prime(\Box)}{\left(\frac{\beta}{2}(\alpha(\Box)+1)+ \ell(\Box)\right)\left(\frac{\beta}{2}\alpha(\Box)+ \ell(\Box) + 1\right)} \times \prod_{j=0}^{\ell(\kappa) - 1} \frac{\Gamma\left(\nu+ \frac{\beta}{2}j + 1 - \kappa_j\right)}{\Gamma\left(\nu + \frac{\beta}{2}j + 1\right)}. \label{eq:forresterLaguerre}
\end{multline}
Thus, by making the substitution $x_i \mapsto 2/x_i$ in \eqref{eq:forresterLaguerre}, we arrive at \eqref{eq:finiteNLaguerremoment}.
\end{proof}

\begin{proof}[Proof of Proposition \ref{LaguerreProp}]
First, observe that    
\begin{align*}
G_{N,\beta}(\nu,r)=\hat{\mathcal{E}}_{N,\beta}^{(\nu)}\left[\left(x_1^{(N)}+\cdots+x_N^{(N)}\right)^r\right],
\end{align*}
since all the points are non-negative. The convergence statement is then a consequence of Proposition \ref{TraceMomentsConv} (recall all points are non-negative) by virtue of Lemma \ref{InverseLaguerreConsistency} and the fact that $\hat{\mathcal{E}}_{N,\beta}^{(\nu)}\left[|x_1^{(1)}|^r\right]<\infty$ for any $0\le r <\nu+1$, which can be chosen so that $r\ge 1$ since $\nu>0$, where we denote by $\mathsf{Y}_\beta(\nu)$ the limiting random variable $\mathsf{T}_\infty$. 

To obtain the explicit formula (\ref{eq:InvLaguerreExplicit}) we substitute the evaluation \eqref{eq:finiteNLaguerremoment} into the right-hand side of \eqref{EvenMomentsFormula} and use the following fact. For any partition $\kappa$ with $|\kappa| = r$, we have
\begin{equation}
    1= \frac{1}{r!} \sum_{k=0}^r (-1)^{r-k}\binom{r}{k}\prod_{\Box \in \kappa} \left(\frac{\beta}{2} \alpha^\prime(\Box) + k - \ell^\prime(\Box) \right).
\end{equation}
This is a special case of the identity $1 = \frac{1}{r!}\sum_{k=0}^r (-1)^{r-k} \binom{r}{k} p(k)$, valid for any monic polynomial $p$ with $\deg p = r$, which can be seen as follows. 
  For a polynomial $p$, we define $(\Delta p)(x) := p(x+1) - p(x)$. Note that $(\Delta^np)(x) = \sum_{k=0}^n (-1)^{n-k} \binom{n}{k} p(x+k)$, and if $p(x)$ is monic with $\deg p = r$, then $(\Delta p)(x) = rx^{r-1} + \text{lower order}$. Hence, $(\Delta^r p)(0) = r!$ and the desired identity follows. Alternatively, the explicit formula \eqref{eq:InvLaguerreExplicit} follows by dividing both sides of \eqref{eq:finiteNLaguerremoment} by $N^r$ and taking $N \to \infty$.
\end{proof}
\bibliographystyle{siam}
\bibliography{references}

\begin{thebibliography}{10}

\bibitem{Anderson}
{\sc G.~W. Anderson}, {\em A short proof of {S}elberg’s generalized beta
  formula}, Forum Mathematicum, 3 (1991), pp.~415--417.

\bibitem{AssiotisWishart}
{\sc T.~Assiotis}, {\em {E}rgodic decomposition for inverse {W}ishart measures
  on infinite positive-definite matrices}, Symmetry, Integrability and
  Geometry: Methods and Applications, 15 (2019).

\bibitem{Distinguished}
{\sc T.~Assiotis, B.~Bedert, M.~A. Gunes, and A.~Soor}, {\em On a distinguished
  family of random variables and {P}ainlev\'e equations}, Probability \&
  Mathematical Physics, 2-3 (2021), pp.~613--642.

\bibitem{assiotis2020joint}
{\sc T.~Assiotis, J.~P. Keating, and J.~Warren}, {\em On the joint moments of
  the characteristic polynomials of random unitary matrices}, arXiv:
  2005.13961, to appear IMRN,  (2020).

\bibitem{assiotis2020boundary}
{\sc T.~Assiotis and J.~Najnudel}, {\em The boundary of the orbital beta
  process}, Moscow Mathematical Journal, 21 (2021), p.~659–694.

\bibitem{Bailey_2019}
{\sc E.~C. Bailey, S.~Bettin, G.~Blower, J.~B. Conrey, A.~Prokhorov, M.~O.
  Rubinstein, and N.~C. Snaith}, {\em Mixed moments of characteristic
  polynomials of random unitary matrices}, Journal of Mathematical Physics, 60
  (2019), p.~083509.

\bibitem{7authors}
{\sc E.~Basor, P.~Bleher, R.~Buckingham, T.~Grava, A.~Its, E.~Its, and J.~P.
  Keating}, {\em A representation of joint moments of {CUE} characteristic
  polynomials in terms of {P}ainlevé functions}, Nonlinearity, 32 (2019),
  p.~4033–4078.

\bibitem{BasorChenEhrhardt}
{\sc E.~Basor, Y.~Chen, and T.~Ehrhardt}, {\em Painlev\'{e} {V} and
  time-dependent {J}acobi polynomials}, J. Phys. A, 43 (2010), pp.~015204, 25.

\bibitem{Borodin_Olshanski}
{\sc A.~Borodin and G.~Olshanski}, {\em Infinite random matrices and ergodic
  measures}, Communications in Mathematical Physics, 223 (2001), p.~87–123.

\bibitem{GTgraph}
{\sc A.~Borodin and G.~Olshanski}, {\em The boundary of the {G}elfand-{T}setlin
  graph: a new approach}, Adv. Math., 230 (2012), pp.~1738--1779.

\bibitem{BourgadeCircular}
{\sc P.~Bourgade, A.~Nikeghbali, and A.~Rouault}, {\em {Circular Jacobi
  Ensembles and Deformed Verblunsky Coefficients}}, International Mathematics
  Research Notices, 2009 (2009), pp.~4357--4394.

\bibitem{ChenIts}
{\sc Y.~Chen and A.~Its}, {\em Painlev\'{e} {III} and a singular linear
  statistics in {H}ermitian random matrix ensembles. {I}}, J. Approx. Theory,
  162 (2010), pp.~270--297.

\bibitem{ConreyRubensteinSnaith}
{\sc J.~Conrey, M.~Rubinstein, and N.~Snaith}, {\em Moments of the derivative
  of characteristic polynomials with an application to the {R}iemann zeta
  function}, Communications in Mathematical Physics, 267, 611-629,  (2006).

\bibitem{ConreyHardysFunction}
{\sc J.~B. Conrey}, {\em {The fourth moment of derivatives of the {R}iemann
  zeta-function}}, The Quarterly Journal of Mathematics, 39 (1988), pp.~21--36.

\bibitem{dalborgo2019}
{\sc M.~Dal~Borgo, E.~Hovhannisyan, and A.~Rouault}, {\em {M}od-{G}aussian
  convergence for random determinants}, Annales Henri Poincaré, 20 (2019),
  pp.~259--298.

\bibitem{Dehaye2008}
{\sc P.-O. Dehaye}, {\em Joint moments of derivatives of characteristic
  polynomials}, Algebra Number Theory 2, 31–68,  (2008).

\bibitem{Dehaye2010note}
{\sc P.-O. Dehaye}, {\em A note on moments of derivatives of characteristic
  polynomials}, 22nd International Conference on Formal Power Series and
  Algebraic Combinatorics, 681-692, Discrete Math. Theor. Comput. Sci. Proc.
  AN,  (2010).

\bibitem{Dixon}
{\sc A.~L. Dixon}, {\em Generalizations of {L}egendre’s formula
  $ke'-(k-e)k'=\frac{1}{2}\pi$}, Proceedings of the London Mathematical
  Society, 3 (1905), pp.~206--224.

\bibitem{DumitriuEdelman}
{\sc I.~Dumitriu and A.~Edelman}, {\em Matrix models for beta ensembles},
  Journal of Mathematical Physics, 43 (2002), pp.~5830--5847.

\bibitem{ForresterBook}
{\sc P.~J. Forrester}, {\em Log-gases and random matrices}, Princeton
  University Press,  (2010).

\bibitem{forrester2020joint}
{\sc P.~J. Forrester}, {\em Joint moments of a characteristic polynomial and
  its derivative for the circular $\beta$-ensemble}, arXiv:2012.08618,  (2020).

\bibitem{ForresterWitteNagoya}
{\sc P.~J. Forrester and N.~S. Witte}, {\em Application of the
  {$\tau$}-function theory of {P}ainlev\'{e} equations to random matrices:
  ${P}_{VI}$, the {JUE}, {C}y{UE}, c{JUE} and scaled limits}, Nagoya Math. J.,
  174 (2004), pp.~29--114.

\bibitem{ForresterWitte2006}
{\sc P.~J. Forrester and N.~S. Witte}, {\em Boundary conditions associated with
  the {P}ainlev{\'{e}} {III'} and {V} evaluations of some random matrix
  averages}, Journal of Physics A: Mathematical and General, 39 (2006),
  pp.~8983--8995.

\bibitem{AlperDissertation}
{\sc M.~A. Gunes}, {\em On the joint moments of characteristic polynomials from
  the classical compact groups}, in preparation,  (2022).

\bibitem{HallHardysFunction}
{\sc R.~R. Hall}, {\em On the stationary points of {H}ardy's function $z(t)$},
  Acta Arithmetica, 111 (2004), pp.~125--140.

\bibitem{HuaBook}
{\sc L.~Hua}, {\em Harmonic analysis of functions of several complex variables
  in the classical domains}, Chinese edition: Peking, Science Press (1958),
  English edition: Transl. Math.Monographs 6, RI Providence, American
  Mathematical Society, 1963.

\bibitem{HughesThesis}
{\sc C.~Hughes}, {\em On the characteristic polynomial of a random unitary
  matrix and the {R}iemann zeta function}, PhD Thesis, University of Bristol,
  (2001).

\bibitem{kallenberg2002foundations}
{\sc O.~Kallenberg}, {\em Foundations of modern probability}, Probability and
  its Applications (New York), Springer-Verlag, New York, second~ed., 2002.

\bibitem{KeatingSnaith}
{\sc J.~P. Keating and N.~C. Snaith}, {\em Random matrix theory and
  {$\zeta(1/2+it)$}}, Communications in Mathematical Physics, 214 (2000),
  pp.~57--89.

\bibitem{KilipNenciu}
{\sc R.~Killip and I.~Nenciu}, {\em {Matrix models for circular ensembles}},
  International Mathematics Research Notices, 2004 (2004), pp.~2665--2701.

\bibitem{Li_Valko}
{\sc Y.~Li and B.~Valkó}, {\em Operator level limit of the circular {J}acobi
  $\beta$-ensemble}, arXiv: 2108.11039,  (2021).

\bibitem{Neretin_2002}
{\sc Y.~A. Neretin}, {\em Hua-type integrals over unitary groups and over
  projective limits of unitary groups}, Duke Mathematical Journal, 114 (2002).

\bibitem{Neretin_Triangle}
{\sc Y.~A. Neretin}, {\em Rayleigh triangles and non-matrix interpolation of
  matrix beta integrals}, Sbornik: Mathematics, Vol. 194, No. 4,  (2003).

\bibitem{Olshanski_Vershik}
{\sc G.~Olshanski and A.~Vershik}, {\em Ergodic unitarily invariant measures on
  the space of infinite {H}ermitian matrices}, Contemporary Mathematical
  Physics, American Mathematical Society Ser. 2, 175, 137-175,  (1996).

\bibitem{Petrov}
{\sc L.~Petrov}, {\em The boundary of the {G}elfand-{T}setlin graph: new proof
  of {B}orodin-{O}lshanski's formula, and its {$q$}-analogue}, Mosc. Math. J.,
  14 (2014), pp.~121--160, 171.

\bibitem{Pickrell}
{\sc D.~Pickrell}, {\em Measures on infinite dimensional {G}rassmann
  manifolds}, Journal of Functional Analysis, Vol. 70, Issue 2, 323-356,
  (1987).

\bibitem{Qiu}
{\sc Y.~Qiu}, {\em Infinite random matrices and ergodic decomposition of finite
  or infinite {H}ua-{P}ickrell measures}, Advances in Mathematics, Vol. 308,
  1209-1268,  (2017).

\bibitem{RamirezRider2009}
{\sc J.~A. Ramírez and B.~Rider}, {\em Diffusion at the random matrix hard
  edge}, Communications in Mathematical Physics, 288 (2009), pp.~887--906.

\bibitem{Valko_Virag}
{\sc B.~Valk{\'o} and B.~Vir{\'a}g}, {\em The ${\text {sine}}_{\beta }$
  operator}, Inventiones mathematicae, 209 (2016), pp.~275--327.

\bibitem{Winn_2012}
{\sc B.~Winn}, {\em Derivative moments for characteristic polynomials from the
  {CUE}}, Communications in Mathematical Physics, 315 (2012), p.~531–562.

\bibitem{ForresterWitte2000}
{\sc N.~S. Witte and P.~J. Forrester}, {\em Gap probabilities in the finite and
  scaled {C}auchy random matrix ensembles}, Nonlinearity, 13 (2000),
  pp.~1965--1986.

\end{thebibliography}

\bigskip
\noindent
{\sc School of Mathematics, University of Edinburgh, James Clerk Maxwell Building, Peter Guthrie Tait Rd, Edinburgh EH9 3FD, U.K.}\newline
\href{mailto:theo.assiotis@ed.ac.uk}{\small theo.assiotis@ed.ac.uk}

\bigskip
\noindent
{\sc Mathematical Institute, Andrew Wiles Building, University of Oxford, Radcliffe
Observatory Quarter, Woodstock Road, Oxford, OX2 6GG, UK.}\newline
\href{mailto:mustafa.gunes@st-hildas.ox.ac.uk}{\small mustafa.gunes@st-hildas.ox.ac.uk}

\bigskip
\noindent
{\sc Mathematical Institute, Andrew Wiles Building, University of Oxford, Radcliffe
Observatory Quarter, Woodstock Road, Oxford, OX2 6GG, UK.}\newline
\href{mailto:arun.soor@maths.ox.ac.uk}{\small arun.soor@maths.ox.ac.uk}

\end{document}